\documentclass[10pt,a4paper]{amsart}
\usepackage{graphicx,latexsym,color,amssymb,amscd}
\usepackage[all]{xy}
\theoremstyle{plain}
\newtheorem{theorem}{Theorem}[section]
\newtheorem{lemma}[theorem]{Lemma}
\newtheorem{proposition}[theorem]{Proposition}
\newtheorem{corollary}[theorem]{Corollary}
\theoremstyle{definition}

\theoremstyle{remark}
\newtheorem{rem}[theorem]{Remark}

\numberwithin{equation}{section}
\numberwithin{figure}{section}
\newcommand{\bd}{\begin{description}}
\newcommand{\ed}{\end{description}}
\newcommand{\ba}{\begin{array}}      \newcommand{\ea}{\end{array}}
\newcommand{\bc}{\begin{center}}     \newcommand{\ec}{\end{center}}
\newcommand{\be}{\begin{enumerate}}  \newcommand{\ee}{\end{enumerate}}
\newcommand{\beq}{\begin{eqnarray}}  \newcommand{\eeq}{\end{eqnarray}}
\newcommand{\beQ}{\begin{eqnarray*}} \newcommand{\eeQ}{\end{eqnarray*}}
\newcommand{\bi}{\begin{itemize}}    \newcommand{\ei}{\end{itemize}}

\newcommand{\ov}{\overline}

\newcommand{\nbpt}{18}
\newcommand{\figtotext}[3]{\begin{array}{c}\includegraphics{#3}\end{array}}
\newcommand{\rbis}{\figtotext{\nbpt}{\nbpt}{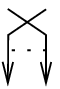}}

\newcommand{\ra}{\rightarrow}
\newcommand{\hra}{\hookrightarrow}
\newcommand{\xra}{\xrightarrow}

\newcommand{\inv}{\mathcal{\iota}}

\newcommand{\Teich}{Teichm\"uller }
\newcommand{\Poin}{Poincar\'e }

\def\ra{{\rightarrow}}

\def\bZ{{\mathbb Z}}
\def\bQ{{\mathbb Q}}

\def\inv{{^{-1}}}

\def\T1g{{\mathcal T}_{g,1}}
\def\M1g{{\mathcal M}_{g,1}}

\def\I1g{{\mathcal I}_{g,1}}

\newcommand\Mor{\operatorname{Mor}}

\newcommand\Aut{\operatorname{Aut}}

\newcommand\Pt{\mathfrak{Pt}}

\newcommand\mc{\mathcal}

\newcommand{\ftt}[1]{\begin{array}{c}\includegraphics{#1}\end{array}}

\renewcommand{\bigcirc}{\!\!\!\ftt{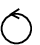}\!\!\!}

\addtocounter{MaxMatrixCols}{5}

\newcommand{\mb}{\mathbf}

\begin{document}
%
%
%
\title[Finite type Invariants and Fatgraphs ]{Finite type  Invariants and Fatgraphs}
\author[J.E. Andersen]{J{\o}rgen Ellegaard Andersen}
\address{Center for the Topology and Quantization of Moduli Spaces\\
Department of Mathematics\\
Aarhus University\\
DK-8000 Aarhus C, Denmark\\}
\email{andersen{\char'100}imf.au.dk}
\author[A.J. Bene]{Alex James Bene}
\address{Departments of Mathematics \\
University of Southern California\\
Los Angeles, CA 90089\\
USA\\}
\email{bene{\char'100}usc.edu}
\author[J.B. Meilhan]{Jean-Baptiste Meilhan}
\address{Institut Fourier \\
Universit\'e Grenoble 1 \\
38402 St Martin d'H\`eres\\
France}
\email{jean-baptiste.meilhan@ujf-grenoble.fr}
\author[R.C. Penner]{R. C. Penner}
\address{Departments of Mathematics and Physics/Astronomy\\
University of Southern California\\
Los Angeles, CA 90089\\
USA\\
~{\rm and}~Center for the Topology and Quantization of Moduli Spaces\\
Department of Mathematics\\
Aarhus University\\
DK-8000 Aarhus C, Denmark\\}
\email{rpenner{\char'100}math.usc.edu}
\begin{abstract}
We define an invariant $\nabla _G(M)$ of pairs $M,G$, where $M$ is a
3-manifold obtained by surgery on some framed link in
the cylinder $\Sigma\times I$, $\Sigma$ is a connected surface with at least
one boundary component, and $G$ is a fatgraph spine of $\Sigma$.
In effect, $\nabla_G$ is the composition with the $\iota_n$ maps of Le-Murakami-Ohtsuki of the link invariant of 
Andersen-Mattes-Reshetikhin computed relative to choices determined by the fatgraph $G$; this provides a
basic connection between 2d geometry and 3d quantum topology.
For each fixed $G$, this invariant is shown to be universal for homology cylinders,
i.e., $\nabla _G$ establishes an isomorphism from an appropriate vector space $\overline{\mathcal H}$
of homology cylinders to a certain algebra of Jacobi diagrams.  Via composition
$\nabla_{G'}\circ\nabla _G^{-1}$ for any pair of fatgraph spines $G,G'$ of $\Sigma$, we derive
 a representation of the Ptolemy groupoid, i.e., the combinatorial model for the fundamental path groupoid of Teichm\"uller space, as a group of automorphisms of this algebra. 
  The space $\overline{\mathcal H}$ comes equipped 
 with a geometrically natural product induced by stacking cylinders on top of one another
 and furthermore supports related operations which arise by gluing a homology 
 handlebody to one end of a cylinder or to another homology handlebody.
 We compute how $\nabla _G$ interacts with all three operations explicitly in terms of natural 
 products on Jacobi diagrams and certain diagrammatic constants.  Our main result gives an explicit extension of the 
 LMO invariant of 3-manifolds to the Ptolemy groupoid in terms of these operations, and this groupoid extension nearly 
 fits the paradigm of
a TQFT.  We finally re-derive the Morita-Penner cocycle representing the first Johnson homomorphism
using a variant/generalization of $\nabla_G$.
\end{abstract}

\maketitle

\section{Introduction}\label{sec:intro}

In \cite{LMO}, Le, Murakami and Ohtsuki constructed an invariant $Z^{LMO}(M)$
of a closed oriented $3$-manifold $M$ from the Kontsevich integral $Z$ (see $\S$\ref{sect:Kontsevich}) of a framed 
link with $k$ components, where $Z$ takes values in the  space ${\mathcal A}(\bigcirc^k)$ of Jacobi diagrams with core $\bigcirc^k$,   a 
collection of $k$ oriented circles  (see $\S$\ref{defjac}).  
The Kontsevich integral $Z$ is universal among rational-valued Vassiliev invariants, i.e., any 
other factors through it.
The LMO invariant $Z^{LMO}(M)\in \mathcal{A}(\emptyset)$ takes values in Jacobi diagrams with empty core and  arises as 
a suitably normalized post-composition of $Z$ with mappings $$\iota_n: \mathcal{A}(\bigcirc^k)\ra 
\mathcal{A}(\emptyset),$$
which are of key importance for LMO and effectively ``replace circles by sums of  trees'' (see \S\ref{sec:iota}).
The LMO invariant is universal among rational-valued finite type invariants of integral and of rational homology spheres.

In \cite{AMR}, Mattes, Reshetikhin and the first-named author defined  a universal Vassiliev  invariant of links (see 
$\S$\ref{amr} for a partial review) in the product manifold $\Sigma\times I$, where $\Sigma=\Sigma_{g,n}$ is a fixed 
oriented surface of genus $g\ge 0$ with $n\ge 1$ boundary components and $I$ is the closed unit interval, which 
generalizes the Kontsevich integral.  Actually, the determination of this AMR invariant depends on a certain 
decomposition of the surface $\Sigma$ into polygons.

In \cite{penner,penner93,Penner04}, the last-named author described an ideal cell decomposition of the decorated Teichm\"uller 
space of a bordered surface in terms of marked fatgraphs $G$ embedded in $\Sigma$ (see $\S$\ref{reviewfatgraphs} for 
the definitions) and introduced the Ptolemy groupoid $\Pt (\Sigma)$ and its canonical presentation in terms of 
Whitehead moves (see $\S$\ref{ptolemy} for both the moves and the presentation).  A key point is that the natural 
quotient of $\Pt(\Sigma )$ contains the mapping class group $MC(\Sigma )$ of $\Sigma$ as the stabilizer of any object.
A more speculative point (discussed further in $\S$\ref{conclusions}) is that the Whitehead moves which
generate $\Pt(\Sigma)$ may themselves be interpreted as triangulated cobordisms of triangulated surfaces

In fact, the specification of a marked fatgraph $G$ in $\Sigma$ suffices to determine a polygonal decomposition (see $\S$\ref{polydec} for this 
construction) as required for the definition of the AMR invariant.  This is a basic connection between decorated Teichm\"uller theory and finite 
type invariants which we exploit here.

Indeed, we define an invariant $ \nabla _G$ (see $\S$\ref{defnabla} for the definition and Theorem \ref{thm:inv} for 
its invariance)  taking values in the space $\mc{A}_h$ of $h$-labeled Jacobi diagrams without strut components (see 
$\S$\ref{defjac} for the definitions), where $h=2g+n-1$ is the rank of the first homology group of $\Sigma$ if 
$\Sigma=\Sigma_{g,n}$ has genus $g$ and $n$ boundary components.  
 Specifically, our invariant is defined for any ``cobordism'' $M$, i.e., $\nabla _G(M)$ is defined for any 
 3-manifold $M=(\Sigma \times I)_L$ arising  from Dehn surgery on a framed link $L\subset \Sigma \times I$ and for any 
 marked fatgraph $G$ in $\Sigma$.  In fact, the fatgraph $G$ determines not only the polygonal decomposition necessary 
 for an AMR invariant but also other choices which are required for our new invariant (see $\S$\ref{kpair} for these 
 other choices called systems of  ``latches'' and ``linking pairs'').

The invariant $\nabla _G$ is defined in analogy to $Z^{LMO}$ in the sense that it arises as a suitably normalized 
post-composition of the AMR invariant determined by $G$ with $\iota_n$, so the AMR invariant (actually, a weakened forgetful version of it) plays for 
us the role of the Kontsevich integral in LMO.  We show (see Theorem \ref{thm:univ}) that $\nabla_G$ is universal for 
so-called ``homology cylinders'', which arise for surgeries along a particular class of links called claspers (see 
$\S$\ref{sec:thmunivnabla} for the definitions of homology cylinders and claspers).

Since $\nabla_G$ is universal for homology cylinders, it induces an isomorphism
$$\nabla_G: \overline{\mathcal{H}}_\Sigma\ra \mathcal{A}_h, $$
where $\overline{\mathcal{H}}_\Sigma$ is a quotient of the vector space freely generated by homology cylinders over 
$\Sigma$ (see $\S$\ref{fti} for the precise definition of $\overline {\mathcal H}_\Sigma$).
 It is this manifestation of universality that has useful consequences for the Ptolemy groupoid $\Pt (\Sigma )$ since 
 given two marked fatgraphs $G$ and $G'$ in $\Sigma$, there is the composition
$$\nabla _{G'}\circ \nabla _G^{-1}: {\mathcal A}_h\to{\mathcal A}_h.$$
For essentially formal reasons, this turns out to give a representation
$$\xi :\Pt(\Sigma _{g,1})\to \Aut({\mathcal A}_h)$$
of the Ptolemy groupoid in the algebra automorphism group of ${\mathcal A}_h$.

There are several well-known and geometrically natural operations on $\overline{\mathcal H}$.  Firstly, there is the 
``stacking'' induced by gluing homology cylinders top-to-bottom.  Secondly,  given a homology cylinder over the 
once-bordered surface $\Sigma_{g,1}$ and a genus $g$ homology handlebody (see \S\ref{handlebodies} for the definition), 
we can take their ``shelling product'' by identifying the boundary of the latter with the bottom of the former.  
Thirdly and finally, we can glue two homology handlebodies along their boundaries to get a closed 3-manifold in the 
spirit of Heegaard decompositions which is called the ``pairing'' between the homology handlebodies (see 
\S\ref{stdpairings} for details on all three operations).

In $\S$\ref{pairings} we explicitly define three algebraic maps
$$
\bullet:\mathcal{A}_{2g}\times \mathcal{A}_{2g}\ra \mathcal{A}_{2g},~~
\star:\mathcal{A}_{2g}\times \mathcal{A}_{g}\ra \mathcal{A}_{g},~~
{\rm and}~~   \langle \; , \; \rangle:\mathcal{A}_{g}\times \mathcal{A}_{g}\ra
\mathcal{A}(\emptyset),
$$
which respectively correspond (under conjugation with a normalized version of $\nabla_G$ explained in 
\S\ref{sec:reduced}) to the stacking product, the shelling product, and the pairing (as proved in Theorem 
\ref{thmpairings}).  Furthermore, these operations are computed in terms of  a basic  ``concatenation product'' $\odot$ 
(see $\S$\ref{sec:odot}) with three particular tangles $T_g,R_g,S_g$ (see $\S$\ref{pairings} and Figure \ref{tangles} 
for the definitions of these tangles) respectively corresponding to the three operations; this gives a purely 
diagrammatic interpretation and scheme of computation for each operation.

Our penultimate result relies on a groupoid representation
  \[ \rho: \Pt(\Sigma_{g,1})\ra \mc{A}_{2g}, \]
defined by combining our invariant with a representation from \cite{ABP}, to extend the LMO invariant of integral 
homology spheres to the Ptolemy groupoid in the following sense.
Let $f$ be an element of the Torelli group of $\Sigma_{g,1}$ and let
 $$ G\xra{W_1} G_1\xra{W_2}...\xra{W_k}G_k=f(G) $$
be a sequence of Whitehead moves representing $f$ in the sense of decorated Teichm\"uller theory (see 
$\S$\ref{ptolemy}).  Our  result then states  that the LMO invariant of the integral homology $3$-sphere $S^3_f$ 
obtained by the Heegaard construction via $f$  is given by
\[  Z^{LMO}(S^3_f)=\left\langle v_0 , \left(\rho(W_1)\bullet
\rho(W_2)\bullet \dotsm  \bullet \rho(W_k)\right)\star v_0
\right\rangle\in \mc{A}(\emptyset), \]
where $v_0$ is an explicit diagrammatic constant (see Theorem \ref{thm:lmo} for the precise statement) and the 
operations are fully determined diagrammatically as discussed before.  This formalism shows the sense in which the LMO 
invariant extends to the Ptolemy groupoid as a kind of weakened version of TQFT; whereas the Ptolemy groupoid has not 
made contact with the LMO invariant previously, similar TQFT phenomena and remarks are reported in \cite{MO,CL}.

Finally (in \S\ref{sec:lifttau1}), we use our invariant (actually, a variation/generalization $\nabla_G^{I_g}$ of $\nabla_G$, which takes values in Jacobi diagrams with core $2g$ intervals and depends upon a  ``general system of latches'' $I_g$, in order to associate to the Whitehead move $G\xra {W}G'$ the quotient
$\mc{J}(W)=\nabla_G^{I_g}(\Sigma _{g,1}\times I)/\nabla_{G'}^{I_g}(\Sigma _{g,1}\times I)$) and
define a representation
$$ \mc{J}^\mathsf{Y}\colon \Pt(\Sigma_{g,1})\ra  \Lambda^3 H_1(\Sigma_{g,1};\mathbb{Q}) $$
which coincides with that defined by Morita and Penner \cite{MP} to give a canonical cocycle extension of the first 
Johnson homomorphism \cite{johnson}; see \cite{BKP} for analogous
cocycles extending all of the higher Johnson homomorphisms.
It thus seems reasonable to expect that higher-order calculations should provide a corresponding formula for 
the second Johnson homomorphism and, in light of \cite{morita-casson}, also for the Casson invariant.
In fact, one motivation for the present work was to investigate whether the known extensions to the Ptolemy groupoid of 
the Johnson homomorphisms \cite{MP,ABP,BKP} might be special cases of a more general extension of $Z^{LMO}$, cf.\ 
\cite{GL,Habegger,massuyeau}.

We have learned here that $Z^{LMO}$ indeed extends to the Ptolemy groupoid,
and in particular have derived an explicit purely diagrammatic  extension of $Z^{LMO}$ which is ``nearly a TQFT, '' 
but whose formulas are not particularly simple or natural largely owing to their  dependence upon certain combinatorial 
algorithms from \cite{ABP}.

On the other hand by a related construction (in $\S$\ref{general}), we have in the context of finite type invariants derived an elegant and 
natural Ptolemy groupoid representation which may give a simpler extension of $Z^{LMO}$.  We expect that there is a 
precursor for this in the early days of development of  \cite{MP,ABP,BKP}, where explicit unpleasant formulas were 
ultimately replaced by simpler and more conceptual ones; see $\S$\ref{conclusions} for a further discussion.

\vskip .2cm

\noindent{\it Standard Notation.}~
We shall fix a compact connected and oriented surface $\Sigma=\Sigma_{g,n}$ of genus $g\ge 0$ with $n\ge 1$ boundary 
components, fix a basepoint $p\in \partial\Sigma$ and let $h:= 2g+n-1$ denote the rank of $H_1(\Sigma;\mathbb{Z})$.  We 
shall often write $1_\Sigma=\Sigma\times I$, where $I=[0,1]$.


\section{Definitions}\label{sec:definitions}

\subsection{Jacobi diagrams} \label{jacobi}

We first recall the spaces of diagrams in which the Kontsevich, LMO and our new invariants take values.

\subsubsection{Definitions} \label{defjac}
A \emph{Jacobi diagram} is a finite graph with only univalent and trivalent vertices, or a so-called  ``uni-trivalent'' 
graph, such that each trivalent vertex is equipped with a cyclic ordering of its three incident half-edges. In other 
words, a Jacobi diagram is exactly a uni-trivalent {``fatgraph''} as discussed separately in $\S$\ref{reviewfatgraphs}. 
The \emph{Jacobi degree} or simply \emph{$J$-degree} of a Jacobi diagram is half its number of  vertices.

Let $S=\{s_1,...,s_m\}$ be some finite linearly ordered set and be $X$ a $1$-manifold, where we tacitly assume that $X$ 
is compact and oriented and that its components come equipped with a linear ordering.    A Jacobi diagram $G$  {\it 
lies on} ($X,S$) if the set of univalent vertices of $G$ partitions into two disjoint sets, where elements of one of
these sets are labeled by elements of $S$, and elements of the other  are disjointly embedded in $X$; $X$ is called the 
{\it core} of the Jacobi diagram.
As usual \cite{BN,O} for figures, we use bold lines to depict the $1$-manifold $X$ and dashed ones to depict the Jacobi 
diagram (though fatgraphs will sometimes also be depicted with bold lines), and we take the cyclic ordering at a vertex 
given by the counter-clockwise orientation in the plane of the figure, which is used to determine the ``blackboard 
framing''.

Let $\mathcal{A}(X,S)$ denote the $\mathbb{Q}$-vector space generated by Jacobi diagrams on $(X,S )$, subject to the 
AS, IHX and STU relations depicted in Figure \ref{relations}.
\begin{figure}[!h]
\includegraphics{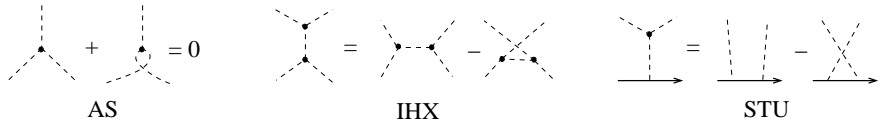}
\caption{The relations AS, IHX and STU. } \label{relations}
\end{figure}
Consider the respective vector subspaces $\mathcal{A}_k(X,S)$ and $\mathcal{A}_{\le k}(X,S)$ generated by Jacobi 
diagrams lying on $(X,S)$ of $J$-degree $k$ and $\le k$, with respective projections of $x\in \mathcal{A}(X,S)$ denoted  
$x_{k}$ and $x_{\le k}$.
Abusing notation slightly, let  $\mathcal{A}(X,S)$ furthermore denote the $J$-degree completion of $\mathcal{A}(X,S)$ 
with its analogous projections to $\mathcal{A}_k(X,S)$ and $\mathcal{A}_{\le k}(X,S)$.  The empty diagram in  
$\mathcal{A}(X,S)$ is often denoted simply $1$. 

We shall primarily be interested in certain  specializations of this vector space:

\vskip .2cm

\leftskip .5cm

\noindent $\bullet$~When  $S=\emptyset$, we write simply $\mathcal{A}(X)=\mathcal{A}(X,\emptyset)$.
If $X$ is the disjoint union of $m$ copies of $S^1$, respectively, $m$ copies of the unit interval, then 
$\mathcal{A}(X)$ is also respectively denoted by $\mathcal{A}(\bigcirc^m)$ and\ $\mathcal{A}(\uparrow^{m})$.  There is 
an obvious surjective ``closing map'' $$\pi: \mathcal{A}(\uparrow^m)\rightarrow \mathcal{A}(\bigcirc^m)$$ which 
identifies to a distinct point the boundary of each component of $X$.

\vskip .1cm

\noindent $\bullet$~ When $X=\emptyset$, we write simply $\mathcal{B}(S)=\mathcal{A}(\emptyset,S)$, called the vector space of 
\emph{$S$-colored Jacobi diagrams}, and when $S=\{1,...,m\}$, we write $\mathcal{B}(m)=\mathcal{B}(S)$, called the 
space of \emph{$m$-colored Jacobi diagrams}.

\leftskip=0ex

\vskip .2cm

\noindent In fact, ${\mathcal A}(X,S)$ has the structure of a Hopf algebra
provided $X=\emptyset, \bigcirc$, or $\uparrow^m$, cf.\ the next 
section.

When $S=\{ s_1,\ldots ,s_m\}$ is a linearly ordered set with cardinality $m$, there is a standard \cite{BN} graded 
isomorphism
  \begin{equation*}
    \chi_S : \mathcal{B}(S)\to \mathcal{A}(\uparrow^m),
  \end{equation*}
called the Poincar\'e--Birkhoff--Witt isomorphism, which maps a diagram to the average of all possible combinatorially 
distinct ways of attaching its $s_i$-colored vertices to the $i^{th}$ interval, for $i=1,\ldots , m$.
When $S=\{1,...,m\}$, we simply write $\chi=\chi_S$;
more generally, given a $1$-manifold $X$ with a submanifold $X'\subset X$ which is isomorphic to and identified with $\uparrow^m$, we have 
the isomorphism
\[ \chi_{X',S}: \mathcal{A}(X- X',S) \to \mathcal{A}(X), \]
which arises by applying $\chi_S$ only to the $S$-labeled vertices.

The \emph{internal degree} or \emph{$i$-degree} of a Jacobi diagram is its number of trivalent vertices.  We call a 
connected Jacobi diagram of $i$-degree zero a {\it strut}, and we denote by $\mathcal{B}^\mathsf{Y}(m)$ the vector space 
generated by $m$-colored Jacobi diagrams without strut components modulo the AS and the IHX relations.  As these two 
relations (unlike STU) are homogeneous with respect to the internal degree, $\mathcal{B}^\mathsf{Y}(m)$ is graded by the 
$i$-degree.  The $i$-degree completion is also denoted $\mathcal{B}^\mathsf{Y}(m)$ and is canonically isomorphic to the 
$J$-degree completion.  

In the rest of this paper, we shall use the simplified notation $\mc{A}_m=\mc{B}^\mathsf{Y}(m)$.
\subsubsection{Operations on Jacobi diagrams} \label{opjac}
There are several basic operations \cite{BN} on Jacobi diagrams as follows:

First of all, disjoint union of  $1$-manifolds $X_1$ and $X_2$ gives a tensor product
$$\otimes\colon  \mc{A}(X_1)\times \mc{A}(X_2)\ra \mc{A}(X_1\sqcup X_2),$$
where the linear ordering on the components of $X_1\sqcup X_2$ is the lexicographic one with components of $X_1$ 
preceding those of $X_2$.  Secondly, if $V_i\subseteq \partial X_i$ for $i=1,2$, and $V_1$ is identified with the reversal of $V_2$ as 
linearly ordered sets of points to form a new $1$-manifold $X$ from $X_1$ and $X_2$, then the {\it stacking product} 
$$\cdot\colon \mc{A}(X_1)\times \mc{A}(X_2)\ra \mc{A}(X)$$ arises by gluing together pairs of identified points and 
combining Jacobi diagrams in the natural way.

Suppose that $Y\subseteq X$ is a connected component of a $1$-manifold $X$.
$Y^{(n)}$ denotes the union of $n$ ordered parallel copies of $Y$.
The \emph{comultiplication map} $\Delta_Y: \mathcal{A}(X)\to  \mathcal{A}(Y^{(2)}\cup X- Y)$
is defined as follows.  Given a diagram $D\in \mathcal{A}(X)$ with $c$ univalent vertices on $Y$, replace $Y$ by 
$Y^{(2)}$ and take the sum of all $2^c$ possible ways of distributing these $c$ vertices to the components of 
$Y^{(2)}$.   More generally, we can recursively define maps
 \[ \Delta^{(n)}_Y: \mathcal{A}(X)\to  \mathcal{A}(Y^{(n)}\cup X- Y) \]
by $\Delta^{(n)}_Y=\Delta_{Y^{(n-1)}_1}\circ\Delta^{(n-1)}_{Y}$, where $Y^{(n-1)}_1$ denotes the first copy of $Y$ in 
$Y^{(n-1)}$.

If $Y \subseteq X$ is a union of components, let  $\ov{Y}$ denote the result of reversing the orientation on $Y$.  The 
\emph{antipode map} $$S_Y: \mathcal{A}(X)\to  \mathcal{A}(\ov{Y}\cup X- Y)$$
is defined by $S_Y(D)=(-1)^c\ov{D}$, where the diagram
$D\in \mathcal{A}(X)$ contains  $c$ univalent vertices attached to $Y$, and $\ov{D}$ arises from $D$  by reversing the orientation of 
$Y$.

\subsection{The Kontsevich integral of framed tangles} \label{sect:Kontsevich}
Let $M$ be a compact connected oriented $3$-manifold whose boundary is endowed with an identification to the boundary 
of the standard cube $C:=[0,1]^3$, and let $X$ be a $1$-manifold possibly with boundary. A \emph{tangle with core $X$ 
in $M$} is a proper embedding of $X$ in $M$ such that all boundary points of $X$ lie on the segments $[0,1]\times
\frac{1}{2}$ in the upper and lower squares $[0,1]^2\times \{1\}$ and $[0,1]^2\times \{0\}$ of $C$.
We shall identify such an embedding with the (isotopy class relative to the boundary) of its image.
A \emph{framed tangle} is a tangle together with a non-vanishing normal vector field.
A \emph{q-tangle} is a framed tangle enhanced with a ``bracketing'', i.e., a consistent collection of
parentheses on each of the naturally linearly ordered sets of boundary points
in the segments $[0,1]\times \{\frac{1}{2}\} \times \{\varepsilon\}$; $\varepsilon=0,1$.

We define two operations on q-tangles in $C$ as follows.
The tensor product $T\otimes T'$ of two q-tangles $T$ and $T'$
is obtained by horizontal juxtaposition and natural bracketing, with $T$ to the left of $T'$ (and reparametrization of 
the ambient cube).
If the upper end of $T$ coincides with the lower end of $T'$, i.e., they coincide as bracketed sets of dots, then
the composition $T\cdot T'$ is obtained by stacking $T'$ on top of $T$ (and reparameterizing the ambient cube).

A fundamental fact \cite{O} is that any q-tangle in $C$ can be (non-uniquely) decomposed as a composition of tensor 
products of (oriented) copies of the elementary q-tangles $I$, $X_{\pm}$, $C_{\pm}$ and $\Lambda_{\pm}$ of Figure 
\ref{elementary}  together with those obtained by orientation-reversal on certain components.
\begin{figure}[!h]
\bc
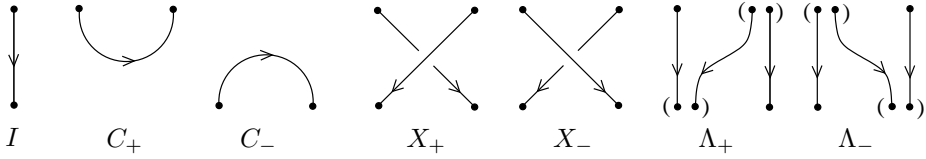
\caption{The elementary q-tangles $I$, $C_{\pm}$, $X_{\pm}$ and $\Lambda_{\pm}$.  } \label{elementary}
\ec
\end{figure}

The \emph{framed Kontsevich integral} $Z(T)$ of a $q$-tangle $T$ with core $X$ in the standard cube $C$ lies in the 
space $\mathcal{A}(X)$ of Jacobi diagrams \cite{BN,O}.  Insofar as $Z(T \cdot T')=Z(T)\cdot Z(T')$ and $Z(T\otimes 
T')=Z(T)\otimes Z(T')$, for any two tangles $T,T'$, it is enough to determine
$Z$ on any tangle by specifying its values on the elementary  q-tangles of Figure \ref{elementary}
by the fundamental fact.
We set $Z(I)=1\in \mathcal{A}(\uparrow)$, and
\begin{equation}\label{Z1}
 \quad Z(C_\pm)=\sqrt{\nu},
 \end{equation}
where $\nu\in \mathcal{A}(\bigcirc)\simeq \mathcal{A}(\uparrow)$ is the Kontsevich integral of the $0$-framed unknot 
(computed in \cite{BNGRT2}).

Recall that $\nu$ is invariant under the antipode map 
and that projecting away the non-strut components of
$\chi(\nu)$ produces zero.

Define
\begin{equation}\label{Z2}
Z(X_\pm)=\textrm{exp}\big( \frac{\pm 1}{2} \rbis \big) = 1 + \sum_{k=1}^{\infty} \frac{(\pm 1)^k}{2^k k!}\big( \rbis
\big)^{k},
\end{equation}
where the $k^{th}$ power on the right-hand side denotes the diagram with $k$ parallel dashed chords.
and set
\begin{equation}\label{Z3}
 Z(\Lambda_\pm)=\Phi^{\pm 1},
\end{equation}
where $\Phi\in \mathcal{A}( \uparrow^3)$ is the choice of an associator (see for example \cite[Appendix D]{O}).

While there are many associators that one may choose to define the Kontsevich integral, we shall
restrict our choice to an even associator (see \cite[$\S$3]{LMpar} for a definition), which necessarily satisfies
  \begin{equation} \label{equevenass}
Z\left( r(T) \right) = r\left( Z(T) \right),
\end{equation}
for any q-tangle $T$ and for any mirror reflection $r$ of its planar projection with respect to any horizontal or 
vertical line \cite{LMpar};
moreover, if $T^{(k)}_i$ is obtained by taking $k$ parallel copies of the $i^{th}$ component of a q-tangle $T$, then we 
have \begin{equation}\label{evencable}
Z(T^{(k)}_i)=\Delta^{(k)}_i(Z(T)).
\end{equation}

\subsection{Fatgraphs}\label{reviewfatgraphs}

A \emph{fatgraph} is a finite graph endowed with a ``fattening'', i.e., a cyclic ordering on each set of half-edges 
incident on a common vertex.  When depicting a fatgraph in a figure, the fattening is given by the counter-clockwise 
orientation in the plane of the figure.
A fatgraph $G$ determines a corresponding ``skinny surface'' with boundary in the natural way, where
polygons of $2k$ sides corresponding to $k$-valent vertices of $G$ have alternating bounding arcs identified in pairs 
as determined by the edges of $G$.
We shall be primarily concerned with the case where such graphs are connected and uni-trivalent with only one univalent 
vertex, and by a slight abuse of terminology, we shall call such a fatgraph a \emph{bordered fatgraph}.   The edge 
incident on the uni-valent vertex of a bordered fatgraph is called the \emph{tail}.

Suppose that $\mb{e}$ is an oriented edge that points towards the vertex $v$ of $G$.  There is a succeeding oriented 
edge $\mb{e}'$ gotten by taking the oriented edge pointing away from $v$ whose initial half edge follows the terminal 
half edge of $\mb{e}$.  A sequence of iterated successors gives an ordered collection of oriented edges starting from 
any oriented edge called a \emph{boundary cycle} of $G$, which we take to be cyclically ordered and evidently 
corresponds to
a boundary component of the associated skinny surface.  We shall call any subsequence of a boundary cycle of $G$ a 
\emph{sector}.  By a \emph{once bordered fatgraph}, we mean a bordered fatgraph with only one boundary cycle, which 
canonically begins from the tail.

Thus, the oriented edges of any once bordered fatgraph $G$ come in a natural linear ordering, namely, in the order of 
appearance in the boundary cycle starting from the tail.
For a connected bordered fatgraph $G$, we can also linearly order the oriented edges by defining the \emph{total 
boundary cycle} as follows.  Let the total boundary cycle begin at the tail and continue until it returns again to the
tail.  If every oriented edge has not yet been traversed, then there is a first oriented edge $\mb{e}$ in this sequence 
such that the oppositely oriented edge $\mb{\bar e}$ has not yet been traversed by connectivity.  We then extend  the 
total boundary cycle by beginning again at $\mb{\bar e}$ and continuing as before until the boundary cycle containing 
$\mb{\bar e}$ has been fully traversed.   By iterating this procedure, we eventually traverse every oriented edge of 
$G$ exactly once.   According to our conventions for figures, the total boundary cycle is oriented with $G$ on its 
left.

Finally, a \emph{marking} of a bordered fatgraph $G$ in a surface $\Sigma=\Sigma_{g,n}$ of genus $g$ with $n>0$ 
boundary components with basepoint $p$ in its boundary,  is a homotopy class of embeddings $G\hra \Sigma$ such that the 
tail of $G$ maps to a point $q\neq p$ on the same component of the boundary of $\Sigma_{g,n}$ as $p$ and the complement 
$\Sigma- G$ consists of a disc (corresponding to the boundary component containing $p$) and $n-1$ annuli (corresponding 
to the remaining boundary components).
The relative version \cite{Penner04} of decorated Teichm\"uller theory \cite{penner} shows that the natural space of all marked 
fatgraphs in a fixed bordered surface is identified with a trivial bundle over its Teichm\"uller space.

\subsection{The polygonal decomposition associated to a fatgraph}\label{polydec}

By a \emph{bigon}, \emph{square} or \emph{hexagon} in a surface $\Sigma$ with boundary, we mean a (topologically) 
embedded closed disc $D^2\hra \Sigma$ such that the intersection $D^2\cap \partial \Sigma$ is the union of one, two or 
three disjoint closed intervals, respectively, called the \emph{bounding edges};  the closures of the components of the 
remainder of $\partial D^2$ are called the \emph{cutting edges}.

Given a marked  bordered fatgraph  $G\hra\Sigma$, its corresponding skinny surface is naturally diffeomorphic to 
$\Sigma$ itself thus providing a polygonal decomposition
\[ \Sigma=B \cup (\cup_i S_i) \cup (\cup_j H_j), \]
where each trivalent vertex corresponds to a hexagon $H_j$, each non-tail edge corresponds to a square $S_i$, and the 
tail corresponds to a bigon $B$, such that the intersection of any two of these components consists of a (possibly 
empty) union of cutting edges.   We refer to any $S_i\times I$ or $B\times I$ as a \emph{box} and to the box $B\times 
I$ associated to the the tail of $G$ as the \emph{preferred box}.  The faces of the boxes corresponding to cutting
edges are called the {\it cutting faces}.  This decomposition of $\Sigma$ is the \emph{polygonal decomposition 
associated to the fatgraph} $G$ marking in $\Sigma$ and is denoted $P_G$.

Such a decomposition $P_G$ of $\Sigma$ into 2-, 4-, and 6-gons, together with a specification of one bounding edge for 
each hexagon, provides sufficient data to define the AMR invariant of \cite{AMR}, which is discussed in the next 
section.  We call the specified bounding edge of each hexagon (as well as the corresponding sector of $G$)  its 
\emph{forbidden sector}.  One can check that for any choice of forbidden sectors for $P_G$, any framed link $L$ in 
$1_{\Sigma}$  can be isotoped  in $1_{\Sigma}$ and endowed with a bracketing of its intersection with the cutting faces 
so that:

\begin{itemize}
\item  For each square $S_i$, the tangle $L^S_i:=L\cap (S_i\times I)$ is a q-tangle in the cube $S_i\times I\cong
    C$ as in $\S$\ref{sect:Kontsevich}.  Similarly, the tangle $L^B:=L\cap (B\times I)$ is a q-tangle in $B\times
    I$.
\item For each hexagon $H_j$, the tangle $L^H_j:=L\cap (H_j\times I)$ is a ``trivial'' q-tangle in the sense that:
    there are no crossings of strands of $L^H_j$;
no strand of $L^ H_j$ connects the two edges of $\partial H_j$ adjacent to the forbidden sector;
the bracketing of the intersection of $L$ with the cutting face opposite the forbidden sector is the concatenation
of the bracketings for the other two cutting faces in the natural way, cf.\ Figure \ref{2torus}.
\item Pairs of bracketings corresponding to the two sides of a single cutting face must coincide (as follows from
    their definition).
\end{itemize}

If a link satisfies these conditions, then we say that it is in \emph{admissible position} with respect to the 
polygonal decomposition $P_G$ associated to the marked fatgraph $G$ in $\Sigma$.
An example is given in Figure \ref{2torus}, where we have labeled each forbidden sector by $\ast$.

\begin{figure}[!h]
\includegraphics{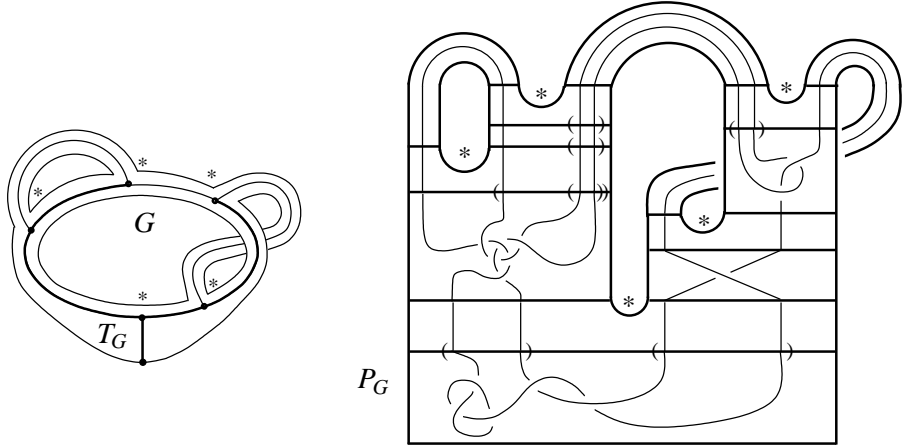}
\caption{A fatgraph $G$ marked in the twice-punctured torus $\Sigma_{1,2}$, the maximal tree $\tau_G$, and a knot in
admissible position with respect to the polygonal decomposition $P_G$.} \label{2torus}
\end{figure}

In fact, a marked fatgraph $G$ in a surface $\Sigma$ not only determines the required polygonal decomposition $P_G$ of 
$\Sigma$ as already discussed, it furthermore determines a collection of forbidden sectors as follows.

By the \emph{greedy algorithm} of \cite{ABP}, there is a canonical maximal tree $\tau_G$ in $G$ built by traversing the 
total boundary cycle of $G$ starting from the tail and ``greedily'' adding every traversed edge  to $\tau_G$ provided the 
resulting graph is simply connected.  See Figure \ref{2torus}.  Note that during this process, the corresponding subset 
of $\tau_G$ is always a connected tree, and that the tail and all vertices of $G$ are included in $\tau_G$.  See 
\cite[$\S$3]{ABP} for a detailed exposition of the greedy algorithm as well as its other manifestations and 
applications.

Given a bordered fatgraph $G$, its  \emph{generators} are the edges in the complement $X_G=G - \tau_G$ of the maximal tree 
$\tau_G$.  Note that  there is a natural linear ordering on the set of generators, and each generator comes equipped with 
an orientation, where the ordering and orientation are determined as the first encountered during the traversal of the 
total boundary cycle.

By general principles about maximal trees, each vertex $v$ of $G$ is connected to the tail by a unique embedded path in 
$\tau_G$, and this path contains a unique edge of $G$ incident to $v$.     As each hexagon of the decomposition of 
$\Sigma$ corresponds to a vertex of $G$, we may define the forbidden sector of a hexagon to be the one opposite the 
edge contained in the path initiating from the corresponding vertex.  See Figure \ref{2torus}.

\begin{lemma}\label{comb}
For any marked fatgraph $G$ in $\Sigma$, the specified forbidden sectors and the corresponding polygonal decomposition 
$P_G$ have the property that any link $L$ in $\Sigma\times I$ can be isotoped so that it intersects each box except the 
preferred one in a trivial q-tangle.
\end{lemma}

\begin{proof}
For the purposes of this proof, we distinguish between the T-boxes, coming from the edges of $\tau_G$, and the G-boxes, 
coming from the edges of $X_G$.
To begin the isotopy, for any G-box $S_x$ containing a non-trivial tangle $L_x^S$, isotope $L_x^S$ out of $S_x$ in 
either direction through the adjacent hexagon and then into an adjacent T-box in a way which avoids producing arcs 
parallel to the forbidden sector of
the hexagon.  This results in a link which is trivial in all G-boxes.  Next, for any T-box $S_t$ containing a 
non-trivial tangle $L_t^S$, similarly isotope $L_t^S$ into a neighboring T-box which is closer to the tail via the path 
in the maximal tree $\tau_G$.  Note that such an isotopy can be performed by our choice of forbidden sectors.  Repeated 
application of this last step results in a link which is trivial in all boxes except the preferred one.
 \end{proof}

\subsection{The AMR invariant}\label{amr}

Andersen, Mattes and Reshetikhin \cite{AMR} defined a universal Vassiliev invariant of links in $1_\Sigma=\Sigma\times 
I$, for $\Sigma$ a surface with boundary, which generalizes the Kontsevich integral; we shall only require a weak
version of their more general construction in this paper.  These invariants depend on the choice of a polygonal 
decomposition of the surface $\Sigma$ together with other essentially combinatorial choices in order to decompose the 
link into suitable sub-links (as for the Kontsevich integral), and these
choices (and more) are provided by a marked fatgraph $G$ in $\Sigma$ as discussed in the previous section.

Let $\mathcal{L}(\Sigma;m)$ be the set of isotopy classes of oriented framed $m$-component links in the thickened 
surface $1_\Sigma$.
Fixing an even associator $\Phi\in \mathcal{A}( \uparrow^3)$ for the Kontsevich integral $Z$ once and for all and 
choosing a fatgraph $G$ marking in $\Sigma$, we  define the {\it AMR invariant}
\[ V_G: \mathcal{L}(\Sigma;m) \to \mathcal{A}(\bigcirc^m) \]
as follows.
Given a link $L\in \mathcal{L}(\Sigma;m)$ in admissible position, we apply $Z$ to each q-tangle $L^B$ and $L^S_i$ and 
map each trivial q-tangle $L^H_j$ to the empty Jacobi diagram in $\mathcal{A}(\uparrow^{|L^H_j|} )$, where $|L^H_j|$ is 
the number of connected components of $L^H_j$.  By choosing an even associator, we do not need to distinguish between 
the top and the bottom of the tangles $L^B$ and $L_i^S$.  We finally compose the resulting Jacobi diagrams as 
prescribed by the polygonal decomposition $P_G$ associated to $G$ to produce the desired $V_G(L)\in 
\mathcal{A}(\bigcirc^m)$.
The invariant depends on the choice of associator and the fatgraph $G$.  We shall also make use of natural extensions 
of  this invariant to certain framed tangles in $1_\Sigma$ with endpoints on $(\partial \Sigma) \times 
\{\frac{1}{2}\}$.

Our
definition of $V_G$ differs from \cite{AMR}  insofar as the original invariant  takes values in sums of diagrams on the 
surface, and we are post-composing with the map that forgets the homotopy data of how these diagrams lie in the surface.
This is of course a 
dramatic loss of information, and we wonder what would be the induced equivalence relation on 
$\mathcal{L}(\Sigma, m)$
assuming faithfulness of the original invariant  \cite{AMR}, which gives not only an isotopy invariant but also a universal Vassiliev invariant of links in $1_\Sigma$.  See $\S$\ref{conclusions} for a further discussion.

\subsection{Linking pairs and latches}\label{kpair}

It is a satisfying point that a marked fatgraph $G$ suffices to conveniently determine the choices required to define 
the AMR invariant $V_G(L)\in \mathcal{A}(\bigcirc^m)$ of an $m$ component link $L$.  The fatgraph furthermore 
determines several other ingredients required for the definition of our new invariants as we finally describe.

Let $M$ be a closed $3$-manifold, possibly with boundary.  A \emph{linking pair} in $M$ is a 2-component link $K$ 
arising from an embedding of a standard torus into $M$, where the first component of $K$ is the core of the torus, 
called the ``longitude'' of the pair,  and the second is a small null-homotopic $0$-framed meridian of it, called the 
``meridian'' of the pair.  We say a link $L$ is \emph{disjoint} from a linking pair $K$ in $M$ if $K$ is a linking pair 
in $M-L$.

In particular in $S^3$, any two framed links $L$ and $L'=L\sqcup K$,  with $K$ a linking pair and $L$ disjoint from $K$, are 
related by Kirby I and Kirby II moves.  However, this is no longer the case for 3-manifolds with boundary, and one must 
introduce a third move, called Kirby III, where a linking pair may be added  or removed from a surgery link without 
changing the resulting 3-manifold.  The precise statement of the theorem of  \cite{roberts} is that surgery on two 
framed links in a 3-manifold with boundary determine homeomorphic 3-manifolds if and only if the two links are related 
by a finite composition of the three Kirby moves, which are sometimes denoted simply KI-III.

Consider the ordered set of generators $X_G=\{x_1,...,x_h\}$ of $G$.
For each $x_i\in X_G$, the two paths from its endpoints to the tail in $\tau_G$ combine with $x_i$ to form a closed loop 
based at the tail.  By construction, these based loops comprise a (linearly ordered) set of generators for the 
fundamental group of $\Sigma$.
Let $l_i$ denote a simple closed curve, representing the free homotopy classes of
the $i$th loop, framed along $\Sigma\times\{ 1\}$ and pushed off in the $I$ direction in $1_\Sigma=\Sigma\times I$ to 
height $1-i\epsilon$, for some small  $\epsilon>0$ fixed independently  of  $i$, and pick a small $0$-framed meridian $m_i$ 
of $l_i$.  This 
provides a collection of linking pairs
$$ K_G:=\cup_i (l_i\cup m_i)\subset 1_\Sigma, $$
called the \emph{system of linking pairs} for $1_\Sigma$ determined by the fatgraph $G$.

\begin{lemma} \label{reducedform}
Let $G$ be a marked fatgraph in $\Sigma$ and let  $L\subset 1_\Sigma$ be a framed link disjoint from $K_G$.
There exists a (non-unique) framed link $L_0$ in $1_\Sigma - K_G$
contained in the preferred box of $P_G$, such that $L\cup K_G$ is equivalent under isotopy and Kirby II moves to 
$L_0\cup K_G$ in $1_\Sigma$.
\end{lemma}

\noindent Such a representative for a framed link as in the previous lemma is called a {\it reduced representative}.

\begin{proof}
By an isotopy supported in a neighborhood of the longitudes of $K_G$, we may arrange that the meridians are all 
contained in the preferred box.
According to Lemma \ref{comb}, we may assume that $L$ is admissible for $G$ and intersects each box except the 
preferred one in a trivial q-tangle.  By Kirby II moves along the meridians, we may arrange that each component of $L$ 
lies in a different slice of $\Sigma\times I$ than the longitudes.  We may furthermore arrange that the link does not meet the box corresponding to any 
generator $X_G$ of $G$ by sequentially, one generator at a time, performing Kirby II moves along the longitudes of 
$K_G$.  Each Kirby II move discussed thus far can and furthermore will be performed using bands for the slides that lie 
within a single box.  A final isotopy of the resulting link produces the desired link $L_0$, and $L\cup K_G$ is 
equivalent under Kirby II and isotopy to $L_0\cup K_G$ in $1_\Sigma$ by construction.
\end{proof}

One final ingredient, which will serve as the core of the space $\mc{A}(\uparrow^h)$ in which our invariant takes its 
values, is also determined by the generators of the fatgraph $G$.
In each box corresponding to a generator of $G$, consider an embedded arc in the boundary of $1_\Sigma$ as depicted in 
Figure \ref{latch}.
\begin{figure}[!h]
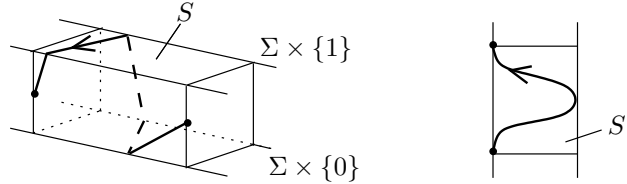 \caption{A latch based at the box $S$ in $1_\Sigma$.  The right-hand side is a projection in the
$I$ direction of $\Sigma\times I$.  } \label{latch}
\end{figure}
Such an arc, called a {\it latch}, is uniquely determined up to relative homotopy by which side of the box contains its 
endpoints, and we determine this side as that corresponding to the first oriented edge traversed by the total boundary 
cycle of $G$.  This collection of latches, one for each generator of $G$, is called  the  \emph{system of latches} 
$I_G$ \emph{determined} by $G$, and they occur in a natural orientation and linear order as before.
(These standard latches determined by the fatgraph admit a natural generalization given in $\S$\ref{general}, which is 
equally well-suited to the construction given in the next section.)
\section{The invariant $\nabla_G$}\label{sec:nabla}
By a \emph{cobordism over $\Sigma$}, we mean a $3$-manifold $(\Sigma \times I)_L$ obtained by surgery on some 
framed link $L$ in $1_\Sigma$.  In particular, a cobordism over $\Sigma$ comes equipped with an identification 
$\partial(\Sigma\times I)\approx \partial (\Sigma \times I)_L$, and two cobordisms are regarded as equivalent if 
there is a diffeomorphism between them that is equivariant for this identification.
Denote by $\mc{C}(\Sigma)$ the set of equivalence classes of cobordisms over $\Sigma$.

\subsection{The invariant $\nabla_G$ of cobordisms}\label{defnabla}
Our construction of $\nabla_G$ is modeled on the LMO invariant $Z^{LMO}$, where the role of the Kontsevich integral is 
now played by the AMR invariant defined in the previous section, and it relies on the LMO maps $\iota_n$, which we next 
recall and slightly extend.
\subsubsection{The map $\iota_n$} \label{sec:iota} This map is a key tool for LMO and for us as well.
It ``replaces circles by sums of trees'' in the rough sense that a core circle component can be erased by suitably 
summing over all trees spanning the endpoints of a Jacobi diagram in that component.

More precisely, for any pair $(m,n)$ of positive integers, any $1$-manifold $X$ without circle components and any 
linearly ordered $S=\{s_1,...,s_m\}$, first define the auxiliary map
$$\aligned j_n&: \mathcal{A}(X,S)\rightarrow \mathcal{A}(X)\cr
j_n(D)&= \left\{ \begin{array}{ll}
O_{n}(<D>) & \textrm{ if $D$ has exactly $2n$ vertices labeled with each color,} \\
0 & \textrm{ otherwise,}
\end{array} \right.
\endaligned$$
where $<D>$ is the sum of all possible Jacobi diagrams obtained by pairwise identifying univalent vertices of $D$ 
having the same color, and $O_{n}$ serially removes all isolated loops, one at a time and each with a compensatory 
factor $(-2n)$.

Given $x\in \mathcal{A}(X\sqcup  \bigcirc^m)$, where $X$ has no circle components and those of
 $\bigcirc^m$ are labeled by $S=\{ 1,\ldots, m\}$ in the natural way,
choose an element $y\in \mathcal{A}(X\sqcup \uparrow^m)$ such that $\pi(y)=x$, and consider
$\chi^{-1}(y)\in \mathcal{A}(X,S)$, where $\chi$ is the Poincar\'e--Birkhoff--Witt isomorphism.
The assignment $\iota_n(x):=\left(j_n(\chi^{-1}(y))\right)_{\le n}$ yields a well-defined map
\[ \iota_n: \mathcal{A}(X\sqcup  \bigcirc^m)\rightarrow \mathcal{A}_{\le n}(X). \]
\noindent Note that the definition given here is small reformulation of a simplified version \cite{Le2} of the original 
\cite{LMO}.
\subsubsection{Definition of the invariant $\widehat{\nabla}_G$}
Let $M$ be a cobordism over $\Sigma$ and let $G$ be a marked bordered  fatgraph in $\Sigma$.   $G$ determines 
the polygonal decomposition $P_G$ of $\Sigma$ with its forbidden sectors, the system $K_G$ of linking pairs and the 
system $I_G$ of latches in $1_\Sigma$, as well as the maximal tree $\tau_G$ and the system $X_G$ of generators.

Take a representative link $L\subset 1_\Sigma$ for $M$ which is disjoint from $I_G$ and $K_G$, so that 
$M=(1_\Sigma)_L=(1_\Sigma)_{L\cup K_G}$.  The linking number of two oriented components $K_1,K_2$
of $L$ in generic position is defined as follows:  project $K_1,K_2$ to 
$\Sigma\approx\Sigma\times \{ 0\}$
and sum over all crossings of the projections a sign $\pm 1$ associated to
each crossing, where the sign
 is positive if and only if the projections of the tangent vectors to the over- and under-crossing 
in this order
agree with the given orientation on $\Sigma$.
For an arbitrary orientation on the link $L\cup K_G$, we denote by $\sigma^{L\cup K_G}_+$, and \ $\sigma^{L\cup 
K_G}_-$, the respective number of positive and negative eigenvalues of its linking matrix,
which are well-defined independent of choices of orientation on components of $L\cup K_G$.

Denote by $V_G$ the AMR invariant determined by $P_G$ and our choice of even associator.  Set
\begin{equation} \label{nabla+}
\widehat \nabla^{G}_n(L):= \frac{\iota_n(\check{V}_G(L\cup K_G\cup I_G))}
{ \iota_n(\check{V}_G(U_{+}))^{\sigma^{L\cup K_G}_+} \iota_n(\check{V}_G(U_{-}))^{\sigma^{L\cup K_G}_-} }\in
\mathcal{A}_{\le n}(\uparrow^{h}),
\end{equation}
where $U_{\pm}$ denotes the $\pm 1$-framed unknot in $1_\Sigma$, and $\check{V}_G(\gamma)$ arises from $V_G(\gamma)$ 
for any framed tangle $\gamma$ by taking connected sum with $\nu$ on each {\sl closed} component, here using that $\{ 1,\ldots, h\}$ is in canonical bijection with $X_G$.

\begin{theorem}\label{thm:inv}
For each $n\ge 1$, the quantity $\widehat\nabla^{G}_n(L)$ defined in (\ref{nabla+}) does not change under Kirby moves 
KI-III  and does not depend on the orientation of $L$.  Thus,  $\widehat\nabla^{G}_n(L)$ is an invariant of the cobordism $(1_\Sigma)_L$.
\end{theorem}
\begin{proof}
The invariance under Kirby I holds for the usual \cite{O} reason: the change in $\iota_n\left(\check{V}_G(L\cup K_G\cup 
I_G)\right)$ under introduction or removal of a $\pm 1$ framed unknot cancels the change in the denominator from
$\sigma^{L\cup K_G}_{\pm}$.

The invariance under KII follows from precisely the same argument as for the LMO invariant, which follows:  First 
observe that an analogue of \cite[Proposition 1.3]{LMO} holds for the AMR invariant: if two links $L$ and $L'$ in 
$1_\Sigma$ differ by a KII move, then $\check{V}_G(L)$ and $\check{V}_G(L')$ are related by a chord KII move, which is 
the move shown in \cite[Figure 6]{LMO}.
This is true  because on one hand, $\check{V}_G$ satisfies (\ref{evencable}) since we have chosen to work with an even 
associator, and on the other hand,  we can always assume (up to isotopy of the link) that each handleslide occurs along 
a band whose projection to $\Sigma$ is contained in a square $S_i$ in the polygonal decomposition $P_G$.  The 
invariance under KII is then shown purely at the diagrammatic level, and comes as a consequent property of the map 
$\iota_n$, whose construction is precisely motivated by its behavior under a chord KII move; see \cite[$\S$3.1]{LMO}.

We note that using KII moves on the meridian components of $K_G$, we can alter any crossing of a longitude with any 
other link component, whence the value of the invariant does not depend on the particular embedding of $K_G$ in $M$ as 
long as $L$ is disjoint from $K_G$ and the homotopy classes of the longitudes of $K_G$ are preserved.

We finally show that invariance under KIII is guaranteed by the presence of the system $K_G$ of linking pairs.   Let 
$L'$ be obtained by adding a linking pair $l\cup m$ to the link $L$, where $m$ is a $0$-framed meridian of the knot 
$l$.  Using the fact that the set of homotopy classes provided by the longitudes of $K_G$ can be represented by a 
system of generating loops for $\pi_1(\Sigma)$, we use KII moves to  successively slide $l$ along longitude components 
of $K_G$ until we obtain a linking pair with longitude null homotopic in $1_\Sigma$ and possibly linked with meridians 
in $K_G$.
We can arrange by isotopy that this linking pair is contained in a $3$-ball in $1_\Sigma$ and can assume by KII moves that it is unlinked with the meridians of $K_G$ in that 3-ball; as noted earlier, any 
such linking pair in a $3$-ball can be removed using Kirby KI-II moves.

Independence from the choice of orientation on $L$ follows from properties of the map $\iota_n$ just as for the LMO 
invariant; see \cite[$\S$3.1]{LMO}.
\end{proof}

Also just as for the LMO invariant, we unify the series $\widehat\nabla^{G}_n$ into a power series invariant by 
setting
\begin{equation}\label{nabla}
\widehat\nabla_G(M):= 1+\left( \widehat\nabla^{G}_1(L) \right)_{1} +  \left( \widehat\nabla^{G}_2(L) \right)_{2} +
\dotsm \in \mathcal{A}(\uparrow^{h})
\end{equation}
in order to define a map
$$ \widehat\nabla_G: \mc{C}(\Sigma)\ra \mathcal{A}(\uparrow^{h}). $$

In the case of a 2-disc $\Sigma_{0,1}$ with the convention that a single edge for the tail is allowed to be a fatgraph $G$, 
$P_G$ is a disk, and both $K_G$ and $I_G$ are empty, then the invariant $\widehat\nabla_G$ exactly coincides with the 
LMO invariant.

Recall that the space of Jacobi diagrams $\mc{A}(\uparrow^h)$ on $h$ intervals is isomorphic to the space $\mc{B}(h)$ 
of $h$-colored Jacobi diagrams via the Poincar\'e--Birkhoff--Witt isomorphism.  Furthermore, there is the projection of 
$\mc{B}(h)$ onto $\mc{A}_h:=\mc{B}^\mathsf{Y}(h)$, and we shall be equally interested in the value our invariant takes in the 
target space $\mc{A}_h$ and hence define
 \[  \nabla_G  \colon \mc{C}(\Sigma) \ra \mc{A}_h,  \]
 where $\nabla_G$ is the composition of $\widehat\nabla_G$ with the projection $\mc{A}(\uparrow^h)\cong\mc{B}(h) \ra 
\mc{A}_h$.  We wonder whether the strut part of $\hat{\nabla}_G(M)$ is related to the homology type of $M$.

\subsection{Universality of $\nabla_G$ for homology cylinders}\label{sec:thmunivnabla}

Homology cylinders are a special class of cobordisms which are important in the theory of finite type 
invariants, cf.\ \cite{Habiro,G}.
In this section, we show that for any marked bordered fatgraph $G$ in the surface $\Sigma$, the invariant $\nabla_G$ of 
cobordisms is universal among rational-valued finite type invariants of homology cylinders over $\Sigma$ in the 
sense of Goussarov and Habiro \cite{Habiro,G}.  We first recall the definition of these objects and review the theory 
of finite type invariants before stating our universality result.
\subsubsection{Claspers and finite type invariants of homology cylinders} \label{fti}

In this section, we briefly review the Goussarov-Habiro theory of finite type invariants for compact oriented 
$3$-manifolds \cite{G,GGP,Habiro}, which essentially generalizes Ohtsuki's theory \cite{ots} for integral homology 
spheres.

A \emph{clasper} $C$ in a $3$-manifold $M$ is an embedding in $M$ of the skinny surface of a (possibly disconnected) 
Jacobi diagram having a framed copy of $S^1$ attached to each univalent vertex.  The copies of $S^1$ are called the 
\emph{leaves} of $C$, the trivalent vertices are called the \emph{nodes} of $C$, and we still call the 4-gons 
associated to the edges of the graph the \emph{edges} of $C$.
We tacitly demand that each connected component of a clasper contains at least one node.
The number of connected components of $C$ is denoted $|C|$, and
its \emph{degree} is the total number of nodes.  A connected clasper of degree $1$ is often called a \emph{Y-graph}.

A clasper $C$ of degree $k$ in $M$ determines a framed link $L(C)$ in $M$, and {\it surgery along $C$} means surgery 
along $L(C)$.   To construct $L(C)$ from $C$, first apply the \emph{edge splitting rule} shown in the left-hand side of 
Figure
\ref{yh} until $C$ becomes a disjoint union of $k$ $Y$-graphs.
Next in a regular neighborhood, replace each $Y$-graph by a $6$-component framed link as shown in the right-hand side 
of Figure \ref{yh}.
\begin{figure}[!h]
\includegraphics{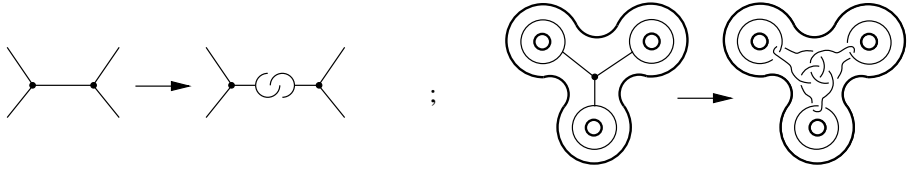}
\caption{The edge splitting rule and the surgery link associated to a $Y$-graph.} \label{yh}
\end{figure}

The link $L(C)$ (that we sometimes also call a clasper) has $6k$ components if $C$ has degree $k$.  The $3k$ components 
coming from the $k$ nodes are called the \emph{Borromean components} of $L(C)$, and the remaining $3k$ components are 
called the \emph{leaf components}.   We may also sometimes write simply $M_C$ for the surgery $M_{C(L)}$.

A \emph{homology cylinder} over a compact surface $\Sigma$ is a 3-manifold  $M=(1_\Sigma)_C$ that arises from surgery on some clasper $C$ in $1_\Sigma$.  
Note that $1_\Sigma=\Sigma\times I$ and hence $M=(1_\Sigma)_C$ comes equipped with embeddings $i^{\pm}: \Sigma \rightarrow M$ with respective images $\Sigma^\pm$, such that:
\begin{itemize}
\item[(i)] $i^+|_{\partial \Sigma}=i^-|_{\partial \Sigma}$;
\item[(ii)] $\partial M=\Sigma^+\cup \left(-\Sigma^-\right)$ and
$\Sigma^+\cap\left(-\Sigma^-\right)=\pm\partial\Sigma^\pm$, where $-\Sigma$ denotes reversal of orientation on
$\Sigma$;
\item[(iii)] $i^\pm_*:H_1(\Sigma;\mathbf{Z})\rightarrow H_1(M;\mathbf{Z})$ are identical  isomorphisms.
\end{itemize}
In the special case where $\Sigma$ has at most one boundary component, such a triple $(M,i^+,i^-)$ satisfying i-iii) conversely always arises from clasper surgery in $1_\Sigma$, cf. \cite{MM}.

The set of homology cylinders over $\Sigma$ up to orientation-preserving diffeomorphism is denoted 
$\mathcal{HC}(\Sigma)$.  There is a natural  {\it stacking product} on $\mathcal{HC}(\Sigma)$ that arises by identifying the top of one homology cylinder with the bottom of another and reparametrizing the interval, i.e., by stacking one clasper on top of another.  This induces a monoid structure on $\mathcal{HC}(\Sigma)$ with $1_\Sigma$ as unit element.

Let $\mathcal{H}_\Sigma$ be the $\mathbb{Q}$-vector space freely generated by elements of $\mathcal{HC}(\Sigma)$ with 
its descending \emph{Goussarov-Habiro filtration} given by
\begin{equation}\label{GHfiltration}
 \mathcal{H}_\Sigma \supset \mathcal{F}_1(\Sigma)\supset\mathcal{F}_2(\Sigma) \supset ...
\end{equation}
where for $k\ge 1$, $\mathcal{F}_k(\Sigma)$ denotes the subspace generated by elements
  $$ [M;C]:=\sum_{C'\subseteq C} (-1)^{|C'|} M_{C'}, $$
with $M\in \mathcal{HC}(\Sigma)$,  $C$ a degree $\ge k$ clasper in $M$, and the sum running over all subsets $C'$ of 
the set of connected components of $C$.

A \emph{finite type invariant of degree $\le k$} is a map $f: \mathcal{HC}(\Sigma)\ra V$, where $V$ is a 
$\mathbb{Q}$-vector space, whose natural extension to $\mathcal{H}_\Sigma$ vanishes on $\mathcal{F}_{k+1}(\Sigma)$.
Denote by $\mathcal{G}_k(\Sigma)$ the graded quotient $\mathcal{F}_k(\Sigma) / \mathcal{F}_{k+1}(\Sigma)$
and let 
\[
\overline{\mathcal{H}}_\Sigma:=~({\rm degree~completion~of}~\mathcal{H}_\Sigma ) /( \cap_k \mathcal{F}_k(\Sigma)).
\]
A fundamental open question is whether $\cap_k \mathcal{F}_k(\Sigma)$ is trivial.
\subsubsection{Universality of the invariant $\nabla_G$}\label{sec:proofuniversalitynabla}
It is known that the LMO invariant is a universal invariant for homology spheres, i.e., every rational-valued finite 
type invariant of homology spheres factors through it \cite{Le3}.  As noted in $\S$\ref{defnabla}, our invariant 
$\nabla_G$ coincides with the LMO invariant for $\Sigma=\Sigma_{0,1}$, and in this section, we prove the following 
generalization of the universality of LMO.
\begin{theorem} \label{thm:univ}
Let $\Sigma$ be a compact connected oriented surface with boundary and $G$ be a marked bordered fatgraph in $\Sigma$.  Then the 
invariant $\nabla_G$ is a universal finite type invariant of homology cylinders over $\Sigma$.
\end{theorem}
\noindent As an immediate consequence we have
\begin{corollary}\label{cor:iso}
For each marked fatgraph $G$ in the surface $\Sigma$, there is a filtered isomorphism
$ \overline{\mathcal{H}}_\Sigma \xrightarrow{\cong} \mathcal{A}_h$
induced by the universal invariant $\nabla_G$, where $h=2g+n-1$.
\end{corollary}
Indeed, the corollary is simply a re-statement of the theorem, and our proof will proceed by exhibiting and checking 
the isomorphism in Corollary ~\ref{cor:iso}.  This will occupy the remainder of the section and begins with the 
definition of the inverse map to $\nabla_G$.

\subsubsection{The surgery map} \label{surgmap}

The graded quotient $\mathcal{G}_k(\Sigma)$ is generated by elements $[1_\Sigma;C]$, where $C$ is a degree $k$ clasper 
in $1_\Sigma$ since
\[ [M;C\cup C'] = [M;C] - [M_{C'};C], \]
where $M\in \mathcal{HC}(\Sigma)$ and $C\cup C'$ is a disjoint union of claspers in $M$ with $C'$ connected.
Define a filtration
\[ \mathcal{G}_k(\Sigma)=\mathcal{F}_{k,3k}(\Sigma)\supset \mathcal{F}_{k,3k-1}(\Sigma)\supset ... \supset 
\mathcal{F}_{k,1}(\Sigma)\supset \mathcal{F}_{k,0}(\Sigma), \]
where
$\mathcal{F}_{k,l}(\Sigma)$ is generated by elements $[1_\Sigma;C]$ with $C$  a degree $k$ clasper in $M$ having $\le 
l$ leaves.  We also  set
\[
\mathcal{G}_{k,l}(\Sigma) := \mathcal{F}_{k,l}(\Sigma) / \mathcal{F}_{k,l-1}(\Sigma).
\]

Denote by $\mathcal{B}^\mathsf{Y}_k(h)$ the $i$-degree $k$ part of $\mc{A}_h=\mathcal{B}^\mathsf{Y}(h)$ and denote by 
$\mathcal{B}^\mathsf{Y}_{k,l}(h)$ the subspace generated by Jacobi diagrams of $i$-degree $k$ with $l$ univalent vertices.  Note 
that $\mathcal{B}^\mathsf{Y}_k(h)=\bigoplus_{0\le l\le 3k} \mathcal{B}^\mathsf{Y}_{k,l}(h)$.

For any marked fatgraph $G\hra \Sigma$ and for any pair $k,l$ of integers with $k\ge 1$ and $0\le l\le 3k$, we define a 
surgery map $\phi^G_{k,l}$ using claspers as follows.
Let $D\in \mathcal{B}^\mathsf{Y}_{k,l}(h)$ be some Jacobi diagram.
For each univalent vertex $v$ of $D$ labeled by $i$, consider an oriented framed knot in $1_\Sigma$ which is a parallel 
copy of the longitude $l_i$ of the system $K_G$ of linking pairs.  For each trivalent vertex of $D$, consider an embedded oriented disk in $1_\Sigma$.  These choices are 
made subject to the constraint that the resulting annuli and disks are pairwise disjoint in $1_\Sigma$.
Connect these various embedded annuli and disks by disjoint bands as prescribed by the diagram $D$ in a way which is 
compatible with their orientations and such that the cyclic order of the three attached bands (given by the 
orientation) at each disk agrees with the cyclic order at the corresponding trivalent vertex of $D$.
The resulting surface is a degree $k$ clasper with $l$ leaves in $1_\Sigma$ denoted $C(D)$, and we set
\[ \phi^G_{k,l}(D) := [1_\Sigma ; C(D)] \in \mathcal{G}_{k,l}(\Sigma). \]

It follows from \cite[Theorem 1]{brane} (see also \cite{Habegger}) that this assignment yields a well-defined
surjection
\[ \phi^G_{k,l}:{\mathcal B}^\mathsf{Y}_{k,l}({h})\to \mathcal{G}_{k,l}(\Sigma). \]
The proof makes use of the calculus of claspers ; see \cite{CHM,GGP,Habiro,O} for similar results.
\subsubsection{Proof of Theorem \ref{thm:univ}}\label{sec:proofunivnabla}
We need to prove the following two facts:
\begin{enumerate}
\item[Fact (1)] The $i$-degree $\le k$ part of $\nabla_G$ is a finite type invariant of degree $k$.
\item[Fact (2)] For each pair $(k,l)$ with $k\ge 1$ and $0\le l\le 3k$, the invariant $\nabla_G$ induces  the
    inverse to the surgery map $\phi^G_{k,l}$.
More precisely, given a Jacobi diagram $D\in \mathcal{B}^\mathsf{Y}_{k,l}({h})$, we have
\begin{equation}\label{inverse}
   \nabla^G_{k,l}(\phi^G_{k,l}(D))=(-1)^{k} D\in \mathcal{B}^\mathsf{Y}_{k,l}({h}),
\end{equation}
\noindent where $\nabla^G_{k,l}$ denotes the composition of $\nabla_G$ with the projection onto
$\mathcal{B}^\mathsf{Y}_{k,l}({h})$.
\end{enumerate}

In order to prove Fact (1), it is enough to consider an element $[1_\Sigma;C]$, where $C$ is a disjoint union of  $k$ 
Y-graphs in $1_\Sigma$ (by construction using the edge splitting rule) and prove that the minimal $i$-degree of $\nabla 
_G( [1_\Sigma;C])$ is $k$. To this end, we may up to isotopy assume that there are $k$ disjoint $3$-balls contained in
the boxes of the polygonal decomposition $P_G$ corresponding to $G$ which intersect $C(D)$ as depicted on the left-hand 
side of Figure \ref{main}.
Note that
\[ [1_\Sigma;C]=\sum_{C'\subseteq C} (-1)^{|C'|} (1_\Sigma)_{C'}=\sum_{C'\subseteq C} (-1)^{|C'|} (1_\Sigma)_{L_0(C')}, 
\]
where $L_0(C')$ is obtained from the link $L(C)$ by replacing each Borromean linking corresponding to a node of $C- C'$ 
by a trivial linking, so in particular $L_0(C')$ is Kirby equivalent to $L(C')$.
In the computation of $\nabla_G([1_\Sigma;C])$ at lowest $i$-degree, we thus obtain for each node of $C$ a trivalent 
vertex attached to the three corresponding core components.  This follows  \cite{O} from the property  \\[0.3cm]
$\textrm{ }$ \input{Zbor.pstex_t} \\
of the Kontsevich integral and implies that the minimal $i$-degree of $\nabla_G([1_\Sigma;C])$ is $k$ as required to 
prove (1).

Turning our attention now to (2),  let $D\in \mathcal{B}^\mathsf{Y}_{k,l}({h})$ be a Jacobi diagram of $i$-degree $k$ with $l$ 
univalent vertices and consider $\phi^G_{k,l}(D)=[1_\Sigma;C(D)]$.  Denote by $J=\frac{1}{2} (k+l)$ the Jacobi degree 
of $D$.   As before, we can assume that there are $k$ disjoint  $3$-balls in the boxes of $P_G$ each of which 
intersects $C(D)$ as depicted on the left-hand side of Figure \ref{main}, and we can assume that there are a further 
$h$ disjoint $3$-balls that intersect the system $K_G$ of linking pairs as illustrated on the right-hand side of the 
same figure.
  \begin{figure}[!h]
\includegraphics{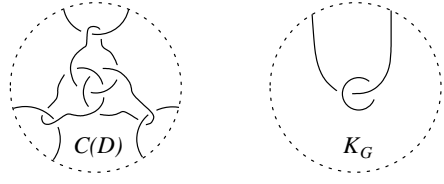}
\caption{Convenient positions for $C(D)$ and $K_G$. } \label{main}
\end{figure}

Let us now compute the (relevant part of the) AMR invariant $\check{V}_G$ of the alternating sum $\phi^G_{k,l}(D)$.
Since we are only computing the lowest $i$-degree part, the contributions of all associators and $\nu$'s can be 
ignored, see (\ref{assocY}) and \cite{BNGRT2} respectively.  In fact, we shall only need to consider the contributions 
arising from the trivalent vertices $(\star)$ and struts coming from crossings, cf.\ $(\star \star)$.
Since the value of $\nabla^G_{k,l}$ is  of $J$-degree $J=\frac{k+l}{2}$,
we must post-compose $\check{V}_G\left( \phi^G_{k,l}(D)\right)$ with the map $\iota_J$ to establish the formula in Fact 
(2).  We need only focus on those terms of $\mathcal{A}(\bigcirc^{6k+2h} \sqcup \uparrow^{h})$ having exactly $2J$ 
vertices on each copy of $S^1$ by definition of $\iota_J$ since only those terms can contribute to the lowest 
$i$-degree part of $\nabla_G(\phi^G_{k,l}(D))$.

We call the core components corresponding to Borromean (leaf, meridian, longitude, latch respectively) components of 
$L\left(C(D)\right)\cup K_G\cup I_{G}$ the Borromean (leaf, meridian, longitude, latch respectively) cores of the 
Jacobi diagrams in the AMR invariant $\check{V}_G(L\left(C(D)\right)\cup K_G\cup I_{G})$.

We first consider contributions of  the linking pairs and recall \cite{O} that \\[0.3cm]
$\textrm{ }\qquad\textrm{}\qquad$ \input{Zedge.pstex_t}. \\
Since the meridian component of a linking pair  is isolated from every component other than its corresponding 
longitude, it follows that all $2J$ vertices on the meridian core must be the ends of distinct struts arising from the 
linking with this longitude.
The resulting connected diagram, which arises from ($\star \star$) with a coefficient $\frac{1}{(2J)!}$, is called a 
\emph{Siamese diagram}; see Figure \ref{AMRclasper}.

Each Borromean component of $L\left(C(D)\right)$  on the one hand forms a Borromean linking with two other such 
components and, on the other hand  links a leaf, cf.\ Figure \ref{main}.  On each Borromean core there is thus one 
vertex arising from ($\star$), and the remaining  $(2J-1)$ vertices are the ends of parallel struts arising from 
($\star \star$) with their opposite ends on a leaf core.

On each leaf core, there is thus only room for one additional vertex.
Furthermore, there is the following dichotomy on leaves of $C(D)$:
\begin{itemize}
\item[(a)] the leaf forms a positive Hopf link with another leaf of $C(D)$ as in the left-hand side of Figure
    \ref{yh};
\item[(b)] the leaf is a parallel copy of a longitude component $l_i$ of $K_G$, pushed off so that it is unlinked
    from the meridian $m_i$, for some $1\le i\le h$.
\end{itemize}
For a type (a) leaf, the only possible contribution is the linking with another type (a) leaf, which produces a strut 
by ($\star \star$).
For a type (b) leaf, we must consider several cases: either the strut comes from a crossing with another type (b) leaf, 
or it comes from a crossing with a component of $I_{G}$ (since a crossing with the longitude components of $K_G$ cannot 
contribute, as we have noted previously).
In the first case, we thus have a strut joining two type (b) leaf cores, and we say that a Jacobi diagram with such a 
strut is \emph{looped}.
In the second case, we have a strut joining the leaf core to a latch core.  A typical example is (partially) 
represented in Figure \ref{AMRclasper}.
\begin{figure}[!h]
\includegraphics{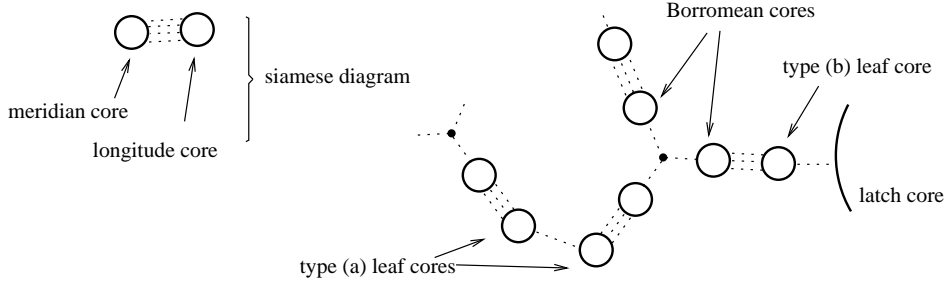}
\caption{Typical lowest $i$-degree term. }\label{AMRclasper}
\end{figure}

By construction, each type (b) leaf has linking number $1$ with exactly one component of $I_{G}$ and $0$ with all 
others.  The lowest $i$-degree terms in $\check{V}_G([1_\Sigma;C(D)])$ are therefore a sum of looped diagrams plus a 
single Jacobi diagram with each type (b) leaf core connected by a strut to a latch core.

We can now apply the map $\iota_J$, and a computation shows that \\[0.3cm]
$\textrm{ }\qquad\textrm{}\qquad\textrm{}\qquad$\input{iota.pstex_t}. \\

We find two (and one respectively) such configurations for each edge incident (and not incident)
on a univalent vertex of $D$, and each comes from ($\star \star$) with a coefficient $\frac{1}{(2J-1)!}$.
This formula also shows that $\iota_J$ maps each Siamese diagram to a factor $(-1)^J (2J)!$, and there are $h=2g+n-1$ 
such diagrams.
We obtain that $\iota_J\left(\check{V}_J(\phi^G_{k,l}(D)\right)$ is given by
\[ (-1)^{(J-1)l+Jh} D + \left\{ \begin{array}{l}
\textrm{terms of $i$-degree $k$ with less than $l$ univalent vertices} \\
 \textrm{terms of $i$-degree $>k$,} \end{array}
\right. \]
\noindent where the terms of $i$-degree $k$ with less than $l$ univalent vertices arise from the looped Jacobi 
diagrams.

To conclude the computation, observe that the surgery link $L\left(C(D)\right)\cup K_G$ satisfies 
$\sigma^{L\left(C(D)\right)\cup K_G}_+=\sigma^{L\left(C(D)\right)\cup K_G}_-=3k+h$.
Since
$\iota_J(\check{Z}(U_{\pm}))=(\mp 1)^J + \textrm{terms of $i$-degree }\ge 1$, we therefore
find\[ \iota_J(\check{V}_P(U_{+}))^{\sigma^{L\left(C(D)\right)\cup K_G}_{+}} 
\iota_J(\check{V}_P(U_{-}))^{\sigma^{L\left(C(D)\right)\cup K_G}_{-}}=(-1)^{Jk+Jh} + ~\textrm{terms of $i$-degree }\ge
1. \]
It follows that
 \[ \nabla^G_{k,l}(\phi^G_{k,l}(D))=(-1)^{k} D, \]
 \noindent which concludes the proof of Theorem \ref{thm:univ}.
\subsection{The rigid $\nabla_G^r$ invariant}\label{sec:reduced}
In this section, we introduce a modified ``rigid'' version $\nabla_G^r$ of our invariant $\nabla_G$, which is formulated in terms of the LMO 
invariant of tangles and again depends on the choice of a marked fatgraph for $\Sigma$.
In this incarnation, $\nabla_G^r$ shares properties with the invariant defined in \cite{CHM}, which gives
an extension of the LMO invariant to so-called Lagrangian cobordisms between once-bordered and closed surfaces.  The invariant in 
\cite{CHM} depends upon choices similar to certain of those determined
by a fatgraph discussed here, and it induces a universal invariant for homology cylinders.
Roughly, it is defined by first ``capping off'' a cobordism by attaching $2$-handles along the boundary producing a 
tangle in a homology ball and then computing the LMO invariant (actually, the equivalent \AA rhus integral) of this 
tangle.

\subsubsection{The rigid $\nabla_G^r$ invariant for homology cylinders} \label{subsec:rednabla}
Let $G$ be a marked fatgraph in the bordered surface $\Sigma$.
Recall that the bigon $B$ in the polygonal decomposition $P_G$ of $\S$\ref{polydec} gives rise to the preferred box in 
$1_\Sigma$, which is identified with the standard cube $C=[0,1]^3$ so that the upper face $f=[0,1]^2\times \{1\}$ is 
the cutting face.
Denote by $F$ a collar neighbourhood $f\times [0,\varepsilon]$ of $f$ in $(1_\Sigma- C)$, where $f$ is identified with 
$f\times \{0\}$, and
fix the standard points $s_i:=\frac{i}{2h+1}\in [0,1]$, for $1\le i\le 2h$.

Consider in $1_\Sigma$ the system of linking pairs $K_G=\sqcup_{i=1}^h (l_i\sqcup m_i)$ determined by $G$ as in 
$\S$\ref{kpair}.
We assume that each $l_i$ lies in the surface $\Sigma\times \{\frac{i}{h+1} \}\subset 1_\Sigma$, for $1\le i\le h$ and 
meets $\partial F$ only at the points $s_j\times \{\frac{i}{h+1} \}\times \{\varepsilon \}$, for $j=i,i+1$, that each 
meridian $m_i$ intersects $\partial F$ only at the points $s_j\times \{\frac{1}{2} \}\times \{0\}$, again for 
$j=i,i+1$, and that $(m_i\cup l_i)\cap (F\cup C)$ is in the standard position depicted in Figure \ref{rigid} up to 
isotopy.
\begin{figure}[!h]
\input{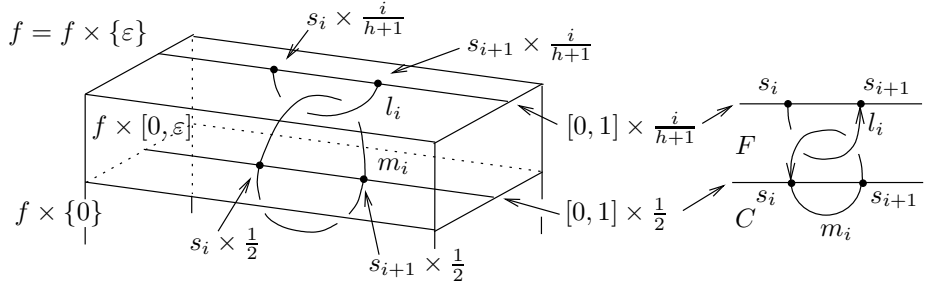}
\caption{System of linking pairs in rigid position. } \label{rigid}
\end{figure}
We say that $K_G$ is in \emph{rigid position} in $1_\Sigma$ in this case.

Suppose that $M=(1_\Sigma)_L$ is a cobordism over $\Sigma$ for some framed link $L$ in $1_\Sigma$.
By Lemma \ref{reducedform}, we can use the system of linking pairs $K_G$ in rigid position to obtain a reduced 
representative $L_0$ which  lies in the preferred box $C$.

Letting $I_G$ denote the system of latches determined by $G$ (cf.\ $\S$\ref{kpair}),
cut $1_\Sigma$ along $f$ in order to split $I_G\cup K_G\cup L_0$ into two q-tangles
\begin{equation}\label{deftgtm}
T_G:= (I_G)\cup (K_G\cap (1_\Sigma- C))
\quad \textrm{ and } \quad
 (K_G\cap C) \cup ( L_0),
\end{equation}
where the bracketing
$(\bullet \bullet) \Big( (\bullet \bullet) \big( (\bullet \bullet)\cdots \left((\bullet \bullet) (\bullet 
\bullet)\right)\cdots\big) \Big)$ is taken on both sets of boundary points.
Set
\begin{equation} \label{nablaredn}
\widehat\nabla^{G,r}_n(L):= \frac{\iota_n(\check{Z}( (K_G\cap C) \cup L_0))}
{ \iota_n(\check{Z}(U_{+}))^{\sigma^{L}_+} \iota_n(\check{Z}(U_{-}))^{\sigma^{L}_-} }\in \mathcal{A}_{\le
n}(\uparrow^{h}),
\end{equation}
where we make use of the same notation as for (\ref{nabla+}), and define the {\it rigid $\widehat \nabla^r_G$ 
invariant} of $M$ to be
\begin{equation}\label{rednabla}
\widehat \nabla^r_G(M):= 1+\big(\widehat \nabla^{G,r}_1(L) \big)_1 + \big(\widehat \nabla^{G,r}_2(L) \big)_2 +\dotsm
\in \mc{A}(\uparrow^h) .
\end{equation}
As before, we define the corresponding \emph{rigid $\nabla^r_G$ invariant}  as the composition of $\widehat\nabla^r_G$ 
with the projection $\mc{A}(\uparrow^h) \ra \mc{A}_h$.

In the next section, we prove the following:
\begin{theorem}\label{thmrednabla}
$\nabla_G^r$ is an invariant of homology cylinders and
induces a graded isomorphism
\[ \nabla^r_G\colon \overline{\mathcal{H}}_\Sigma \rightarrow \mathcal{A}_h \]
for any marked fatgraph $G$ in $\Sigma$.  
\end{theorem}
\noindent Though $\nabla^r_G$ is defined for any cobordism over $\Sigma$, we can at present
only prove it is an invariant of homology cylinders; cf.\ the next section.
\begin{rem}\label{remnablaredlmo}
Recall that the LMO invariant extends naturally to q-tangles in homology balls.  This is done in a similar manner to 
the extension to links in $3$-manifolds \cite{LMO}, and more generally to framed graphs in $3$-manifolds \cite{MO}, via 
a formula similar to (\ref{nablaredn}).  Using this extension, we can reformulate (\ref{rednabla}) as
$$ \widehat\nabla^r_G(M) := Z^{LMO}(B_M,\gamma_M)\in \mathcal{A}(\uparrow^h), $$
where $(B_M,\gamma_M)$ denotes the result of surgery on $(C,K_G\cap C)$ along the link $L_0$.  Note that $B_M$ is indeed 
a homology ball since $M$ is a homology cylinder.
\end{rem}

\subsubsection{Proof of Theorem \ref{thmrednabla}}\label{univtangles}

For $h\ge 1$, let $T_{h,0}=\bigotimes_h C_+$ be the q-tangle in $C$ obtained by horizontal juxtaposition of $h$ copies 
of the q-tangle $C_+$ of Figure \ref{elementary}, with bracketing of the form $(\bullet \bullet) \Big( (\bullet 
\bullet) \big( (\bullet \bullet)\cdots \left((\bullet \bullet) (\bullet \bullet)\right)\cdots\big) \Big)$.  See Figure
\ref{tcap} and set
 \[ \mathcal{T}(h):=\{ (C,T_{h,0})_{\Gamma}\textrm{ : $\Gamma$ is a clasper in $C$ disjoint from $T_{h,0}$} \}. \]
Let $ \mathbb{Q}\mathcal{T}(h)$ denote the vector space freely generated by elements of $\mathcal{T}(h)$.
In analogy to (\ref{GHfiltration}), we have the Goussarov-Habiro filtration
 \[ \mathbb{Q}\mathcal{T}(h)\supset \mathcal{F}_1(h)\supset\mathcal{F}_2(h) \supset ... \]
where $\mathcal{F}_k(h)$  denotes the subspace generated by elements $[(B,\gamma);\Gamma]$ with $(B,\gamma)\in 
\mathcal{T}(h)$ and with $\Gamma$ a degree $\ge k$ clasper in $B$ disjoint from $\gamma$, for $k\ge 1$.  This
filtration serves to define
a notion of finite type invariants for these objects as in $\S$\ref{fti}.
\begin{proposition} \label{thmLMOuniv}
For any $h\ge 1$, the LMO invariant induces a universal finite type invariant for  tangles in $\mathcal{T}(h)$.
\end{proposition}
\begin{proof}
The proof follows closely that of Theorem \ref{thm:univ}.
In particular as in $\S$\ref{surgmap}, we define for each pair $(k,l)$ with $k\ge 1$ and $0\le l\le 3k$ a surgery map 
$\phi^r_{k,l}$ as follows.
Let  $D\in \mathcal{B}^\mathsf{Y}_{k,l}(h)$.
For each $i$-labeled univalent vertex,
pick a parallel copy of a small $0$-framed meridian of the $i$th component of $T_{h,0}$,
and for each trivalent vertex of $D$, pick an embedded oriented disk in $C$.  Connect these meridians and disks by 
disjoint bands as prescribed by the diagram $D$ to obtain a degree $k$ clasper with $l$ leaves denoted $C^r(D)$.
The assignment $\phi^r_{k,l}(D) := [(C,T_{h,0}) ; C^r(D)]$ yields a well-defined surjective map
\[ \phi^r_{k,l}:\mathcal{B}^\mathsf{Y}_{k,l}({h})\to \mathcal{G}_{k,l}(h), \]
where $\mathcal{G}_{k,l}(h)$ is defined as in $\S$\ref{surgmap}.
The rest of the proof follows from the analogues of facts (1) and (2) of $\S$\ref{sec:proofunivnabla}, which hold 
according to exactly the same arguments.
\end{proof}

As a consequence, we have a graded isomorphism
 \[ Z^{LMO}\colon \overline{\mc{T}}(h)\xra{\cong}\mc{A}_{h} \]
induced by the LMO invariant, where
$\overline{\mathcal{T}}(h)$ denotes the quotient $$\bigl ({\rm degree~completion~of}~\mathbb{Q}\mathcal{T}(h)\bigr ) / \left(\cap_{k\ge 1} 
\mathcal{F}_k(h)\right).$$
The inverse isomorphism, denoted $\phi^r$, is induced by the surgery maps $\phi^r_{k,l}$.

We can now proceed with the proof of Theorem \ref{thmrednabla}.
For any marked fatgraph $G$ in $\Sigma$, define a map
\[ J_G: \mathcal{T}(h)\rightarrow \mathcal{HC}(\Sigma) \]
as follows.  If
$(B,\gamma)= (C,T_{h,0})_{\Gamma}\in \mathcal{T}(h)$, where $\Gamma$ is some clasper in $C$ disjoint from $T_{h,0}$,
then $J_G(B,\gamma)$ is the homology cylinder obtained by stacking the tangle $T_G$ defined in (\ref{deftgtm}) above 
$(B,\gamma)$ and performing surgery along the $2h$-component link resulting from this stacking.  We shall give in 
Remark \ref{remgluenabla} a purely diagrammatic version of the map $J_G$ for any marked bordered fatgraph $G$.
As a generalization of the Milnor-Johnson correspondence
of Habegger \cite{Habegger}, we wonder if $J_G$ is invertible; if so, then it would follow that the rigid invariant $\nabla _G^r$ is indeed an invariant not just of homology cylinders but also of general cobordisms over $\Sigma$.

Since $J_G(\cap_{k\ge 1} \mathcal{F}_k(h))\subset \left(\cap_k \mathcal{F}_k(\Sigma)\right)$,
there is an induced map
\[ \overline{J}_G: \overline{\mathcal{T}}(h)\rightarrow \overline{\mathcal{H}}_\Sigma, \]
which  is surjective according to Lemma \ref{reducedform}.

\begin{lemma}\label{lemjG}
The map $\overline{J}_G$ is a graded isomorphism.
\end{lemma}

\noindent This implies Theorem \ref{thmrednabla} since (\ref{rednabla}) can thus be rewritten as 
$\nabla^r_G=Z^{LMO}\circ (\overline{J}_G)^{-1}$.

\begin{proof}[Proof of Lemma \ref{lemjG}]
For each pair $(k,l)$ with $k\ge 1$ and $0\le l\le 3k$, consider the surjective map $J^{k,l}_G: \mathcal{G}_{k,l}(h) 
\rightarrow \mathcal{G}_{k,l}(\Sigma)$ induced by $J_G$.  It suffices to show that $J^{k,l}_G$ is a graded isomorphism,
which follows from commutativity of the following diagram:
\[  \xymatrix{
     \mathcal{B}^\mathsf{Y}_{k,l}(h)
     \ar[d]_{\phi^r_{k,l}} \ar[dr]^{\textrm{ }\phi^G_{k,l}}\\
     \mathcal{G}_{k,l}(h) \ar[r]_{J^{k,l}_G}& \mathcal{G}_{k,l}(\Sigma).
    }
\]
To see that this diagram is in fact commutative,
let $D\in \mathcal{B}^\mathsf{Y}_{k,l}(h)$ and $C^r(D)\subset C$ be the clasper obtained by the construction explained in the 
proof of Proposition \ref{thmLMOuniv}, whence $\phi^r_{k,l}(D)=[(C,T_{h,0});C^r(D)]$.
Applying $J^{k,l}_G$ amounts to stacking the tangle $T_G$ on $(T_{h,0}\cup C^r(D))\subset C$.  The result of this 
stacking is the system of linking pairs $K_G$ in rigid position, together with a clasper with $l$ leaves in $1_\Sigma$, 
each leaf being a disjoint copy of a $0$-framed meridian of the component $m_i$ of $K_G$, for some $i$.
By a Kirby KII move, we can slide each of these leaves along the corresponding longitude component $l_i$ of $K_G$ and 
denote by $\Gamma$ the resulting clasper in $1_\Sigma$.  It follows from the definition of the surgery map $\phi_{k,l}$ 
(see $\S$\ref{surgmap}) that
$[1_\Sigma;\Gamma]=\phi_{k,l}(D)$ in $\mathcal{G}_{k,l}(\Sigma)$ as required.
\end{proof}
\section{Diagrammatic formulations of topological gluings}\label{sec:pairings}
Throughout this section, fix a non-negative integer $g$ as well as a closed genus $g$ surface $\Sigma_g$ which we 
identify with the boundary of the standard genus $g$ handlebody $H_g:= \Sigma_{0,g+1}\times I$, where $\Sigma_{0,g+1}$ 
is a fixed disc in the plane with basepoint on its boundary having   $g$ holes ordered and arranged from left to right.   
We also fix a genus $g$ surface  with one boundary component $\Sigma_{g,1}$ and identify $\Sigma_g=\Sigma_{g,1}\cup 
D^2$ with the closed surface obtained by capping off $\Sigma_{g,1}$ with a disc $D^2$, so that $\partial \Sigma_{g,1}
\subset (\partial \Sigma_{0,g+1})\times I$.
\subsection{Topological operations}\label{gluings}
\subsubsection{Homology handlebodies}\label{handlebodies}
A genus $g$ \emph{homology handlebody} is a 3-manifold $M$ with boundary a closed genus $g$ surface $\Sigma$ such that 
the inclusion $\Sigma\hra M$  induces a surjection in integral homology  with kernel a maximal integral 
isotropic subgroup $\Lambda\subset H_1(\Sigma_g;\bZ)$; in this definition, we always require an identification of the 
boundary $\Sigma$ of $M$ with the fixed surface $\Sigma_g$ and call $\Lambda$ the \emph{Lagrangian} of the handlebody.
 For example, $H_g$ is a homology handlebody, whose associated Lagrangian subspace $\Lambda^{st}$ we call  the 
 \emph{standard Lagrangian} of $\Sigma_g$.

We consider two homology handlebodies $v_1$ and $v_2$ equivalent if there is a diffeomorphism of $v_1$ to $v_2$ which 
restricts to the identity on $\Sigma_g$ under the corresponding identifications.  Denote by $V(\Sigma_g,\Lambda)$ the 
set of equivalence classes of homology handlebodies with Lagrangian $\Lambda$.

By a result of Habegger \cite{Habegger}, any two homology handlebodies are related by clasper surgeries  if and only if 
they have the same induced Lagrangian.
In particular, any genus $g$ homology handlebody with Lagrangian $\Lambda^{st}$ can be obtained by clasper surgery in 
$H_g$.  In other words, we have
$$ V(\Sigma_g,\Lambda^{st})={\mathcal HC}(\Sigma_{0,g+1}). $$
\subsubsection{Stacking, shelling and pairing}\label{stdpairings}
We have already defined in ($\S$\ref{fti}) the natural stacking product for homology cylinders, which induces a map
\[ \cdot \colon \overline{\mc{H}}_{\Sigma_{g,1}}\times \overline{\mc{H}}_{\Sigma_{g,1}}\rightarrow 
\overline{\mc{H}}_{\Sigma_{g,1}}, \]
and we next similarly introduce two further products
$$\aligned
\cup_\iota\colon \overline{\mc{H}}_{\Sigma_{0,g+1}}\times \overline{\mc{H}}_{\Sigma_{0,g+1}}&\ra
\overline{\mc{H}}_{\Sigma_{0,1}},\\
\ast\colon \overline{\mc{H}}_{\Sigma_{g,1}}\times \overline{\mc{H}}_{\Sigma_{0,g+1}}&\rightarrow
\overline{\mc{H}}_{\Sigma_{0,g+1}}.\\
\endaligned$$

The pairing $\cup_\iota$ on the vector space $\overline{\mc{H}}_{\Sigma_{0,g+1}}$ is defined as follows.
Consider the standard orientation-reversing map $\iota\colon \Sigma_g\ra \Sigma_g$ which ``takes longitudes  to 
meridians'' and vice versa so that gluing two copies of $H_g$ along their boundaries via $\iota$ produces the standard 
3-sphere $S^3$.
By gluing two arbitrary handlebodies with boundary $\Sigma_g$ along this map
(i.e., their adjunction space collapsing fibers to points), we obtain a closed 3-manifold and refer to this operation 
as their  \emph{pairing}.  Observe that the pairing of two homology handlebodies in 
$\overline{\mc{H}}_{\Sigma_{0,g+1}}$ is an integral homology 3-sphere, or equivalently, a homology cylinder over 
$\Sigma_{0,1}$, as required.

The  \emph{shelling product} $\ast$ is defined as follows.  Given a genus $g$ homology handlebody $H$ and a homology 
cylinder $(N,i^+,i_-)$ over $\Sigma_g$ in the notation of $\S$ref{fti}, we can glue the boundaries via the identification $\partial H=i^-(\Sigma_{g})\subset 
\partial N$ to obtain a new genus $g$ homology handlebody with boundary $i^+(\Sigma_g)$.
Similarly, given the identification  $\Sigma_g=\Sigma_{g,1}\cup D^2$,  we can glue a homology cylinder $M$ over 
$\Sigma_{g,1}$ to the homology handlebody $H$ to obtain a 3-manifold with boundary $i^+(\Sigma_{g,1})\cup (S^1\times 
I)\cup D^2$.  By gluing a cylinder $D^2\times I$ along $(S^1\times I)\cup D^2$ in the standard way, we obtain a new
genus $g$ homology handlebody $M\ast H$.

To illustrate the shelling product, let $\{a_i,b_i\}_{i=1}^g$, respectively, $\{h_i\}_{i=1}^g$, be the collection of 
disjoint loops in $\Sigma_{g,1}\times I$, respectively, in the genus $g$ handlebody $H_g$, shown in Figure 
\ref{figshell}.  Note that each collection  induces a basis for the first homology group of the corresponding 
3-manifold.
The images of these loops under  the shelling product $H_g=(\Sigma_{g,1}\times I)\ast H_g$, which we still denote by 
$a_i$, $b_i$ and $h_i$, are shown on the right-hand side of the figure.  In particular, note that each $b_i$ is 
null-homotopic in $H_g$, and satisfies $|lk(b_i,h_i)|=1$.
 \begin{figure}[h!]
  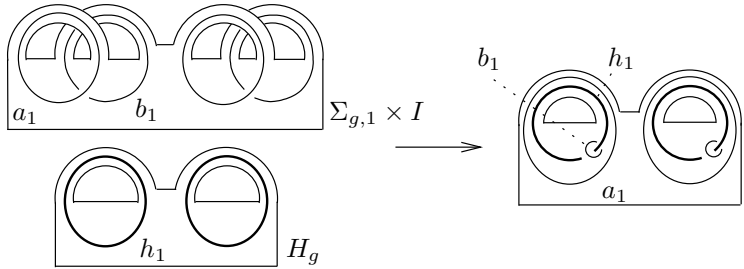
  \caption{The shelling $(\Sigma_{g,1}\times I)\ast H_g$ and the curves $a_i$, $b_i$, $h_i$.
  }\label{figshell}
 \end{figure}

The main goal of this  section is to provide explicit diagrammatic formulas in  $\S$\ref{pairings}  for these three topological operations.
\subsection{A general gluing formula}\label{sec:odot}

We now introduce another more basic operation, which is
a key tool for manipulating our diagrammatic formulas.
\subsubsection{The contraction $\circ$ of labeled Jacobi diagrams}
Let $D\in \mathcal{B}(S)$ and $D'\in \mathcal{B}(S')$ be diagrams, for some finite sets $S$ and $S'$, and let $R\subseteq S\cap 
S'$.
Define the {\it contraction product}
 $D\circ_R D'\in \mathcal{B}\left((S\cup S')- R\right)$, as follows.
If $R=\emptyset$ or if for some $x\in R$ the number of $x$-colored vertices of $D$ and $D'$ is not the same, then set 
$D\circ_R D'=0$, and
otherwise, $D\circ_R D'$ is defined to be the sum of all possible ways of gluing pairwise the univalent vertices of $D$ 
and $D'$ labeled by the same element of $R$.  By linear extension, this defines a contraction map
\[ \circ_R: \mathcal{B}(S)\times \mathcal{B}(S') \to \mathcal{B}\left((S\cup S')- R\right), \]
which we will call the \emph{contraction over $R$}.

Let   $ -_s\in \mathcal{A}\left(\uparrow,\{s\} \right)$ be the Jacobi diagram consisting of a single strut with one 
vertex on $\uparrow$ and one vertex colored by $s$.  Set
$$\lambda(s,u,v): = \chi^{-1}_{\{v\}}\left(\textrm{exp}(-_s)\cdot \textrm{exp}(-_u)\right)\in 
\mathcal{B}(\{s,u,v\}),$$
where the exponential is with respect to the stacking product of Jacobi diagrams.\footnote{
As explained in \cite[Proposition 5.4]{BNGRT1} and \cite[Remark 4.8]{CHM}, $\lambda(s,u,v)$ can also be defined in
terms of the Baker-Cambell-Hausdorff series.}
If $S$, $U$ and $V$ respectively denote the sets $\{s_1,...,s_n\}$, $\{u_1,...,u_n\}$ and $\{v_1,...,v_n\}$, then 
define
\[ \Lambda^n(S,U,V) := \sqcup_{i=1}^{n}  \lambda(s_i,u_i,v_i)\in \mathcal{B}(S\cup U\cup V). \]

\begin{proposition}\label{proplambda} {\rm \cite[Proposition 5.4]{BNGRT1} (see also \cite[Claim 5.6]{CHM})}
For $n\ge 1$, let $D\in \mathcal{A}(X\cup \uparrow^n)$ and $E\in \mathcal{A}(X'\cup \uparrow^n)$, where $X$ and $X'$ 
are two (possibly empty) $1$-manifolds.
Let $D\cdot E\in \mathcal{A}(X\cup X'\cup \uparrow^n)$ be obtained from the stacking product of 
$\mathcal{A}(\uparrow^n)$.  Then
$$\aligned
\chi^{-1}_{\uparrow^n,V} (D\cdot E) &= \Lambda^n(S,U,V)\circ_{S\cup U} \left( \chi^{-1}_{\uparrow^n,S}(D) \sqcup
\chi^{-1}_{\uparrow^n,U}(E) \right) \\
&= \chi^{-1}_{\uparrow^n,S}(D) \circ_{S} \Lambda^n(S,U,V)\circ_{U} \chi^{-1}_{\uparrow^n,U}(E)\\
&\in \mathcal{A}(X\cup
X',V). \\
\endaligned$$
\end{proposition}

\subsubsection{A gluing formula for the LMO invariant of tangles} \label{glueLMO}

For $m,n\ge 0$, denote by $T_{m,n}$ the q-tangle in $C$ represented in Figure \ref{tcap}.  In the natural way, we 
consider the tangles $T_{m,0}$ and $T_{0,n}$ as subtangles of $T_{m,n}$.   \\
 \begin{figure}[h!]
  \includegraphics{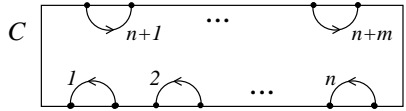}
  \caption{The q-tangle $T_{m,n}$, where the bracketing on both sets of boundary points is of the form
  $(\bullet \bullet) \Big( (\bullet \bullet) \big( (\bullet \bullet)\cdots \left((\bullet \bullet) (\bullet
  \bullet)\right)\cdots\big) \Big) $.  }\label{tcap}
 \end{figure}

Consider the q-tangle $\gamma'=T_{m,n}\cup L'$ in $C$, where $L'$ is some framed link disjoint from $T_{m,n}$.
For any element $D\in \mathcal{A}(\uparrow^n)$, define
\begin{equation}\label{odotk}
D\stackrel{k}{\odot} \gamma' = \chi^{-1}_S D\circ_S \dfrac{\iota_k j^{V}_k \left( \Lambda^{n}(S,U,V)\circ_U
\chi^{-1}_{T_{0,n},U}\hat{Z}(C,\gamma')\right)}{\iota_k(\check{Z}(U_{+}))^{\sigma_+(\gamma')}
\iota_k(\check{Z}(U_{-}))^{\sigma_-(\gamma')}}\in \mathcal{A}(\uparrow^m),
\end{equation}
and set
\begin{equation}\label{odot}
D\odot \gamma' := 1+ \sum_{k\ge 1} (D\stackrel{k}{\odot} \gamma')_k\in \mathcal{A}(\uparrow^m),
\end{equation}
\noindent where
\begin{itemize}
\item $\hat{Z}(C,\gamma')$ is obtained from $\check{Z}(C,\gamma')$ by taking connected sum of $\nu$ with each
    component of $T_{0,n}\subset \gamma'$.
\item $\sigma_\pm(\gamma'):=\sigma_\pm^{L'\cup T_{0,n}}$ denotes the number of positive and negative eigenvalues of
    the linking matrix of the tangle $(T_{0,n}\cup L')\subset \gamma'$.
\item The map $j^{V}_k: \mathcal{B}(S\cup V)\rightarrow \mathcal{B}(S)$ is defined by applying $j_k$ to the
    $V$-colored vertices, as in the definition of $\S$\ref{sec:iota}, and leaving the $S$-colored vertices
    unchanged.
\end{itemize}

It will be useful to define an analogous product for $n$-colored Jacobi diagrams, still denoted by $\odot$. For $E\in 
\mathcal{B}(n)$, we set
\begin{equation}\label{odotcol}
E\odot \gamma' := \chi^{-1}\left( \chi (E)\odot \gamma'\right)\in \mathcal{B}(m).
\end{equation}

We now use this product to give a gluing formula for the LMO invariant.
Let $\gamma=T_{n,0}\cup L$ be a q-tangle in $C$, where $L$ is a framed link disjoint from $T_{n,0}$ so that $b=C_L$ is 
an integral homology ball (in particular, $L$ can be chosen to be a clasper); see Figure \ref{compLMO}.  Set
$$(b,t)=(C,T_{n,0})_L\in \mathcal{T}(n). $$
\begin{figure}[h!]
 \includegraphics{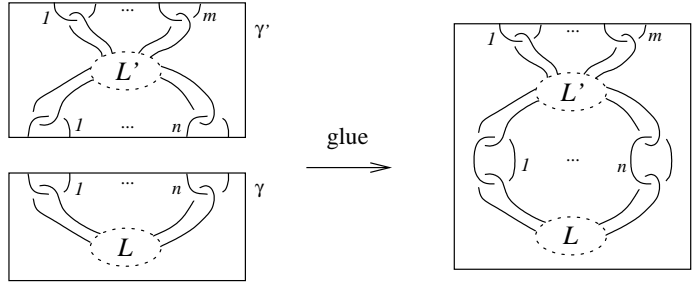}
 \caption{The tangles $\gamma$ and $\gamma'$, and their composition. }\label{compLMO}
\end{figure}

Let $O_n$ denote the $n$ component link arising as the composition of $T_{n,0}$ and $T_{0,n}$, so 
that $\gamma\cdot \gamma'=L'\cup O_n\cup L\cup T_{m,0}$.

\begin{lemma}\label{lemglueLMO}
Let $\gamma$ and $\gamma'$ be two q-tangles as decribed above.
Then the LMO invariant of
\[ (B,T):=\left( C,T_{m,0}\right)_{L\cup O_n\cup L'} \]
is given by
$$ Z^{LMO}(B,T)=Z^{LMO}(b,t)\odot \gamma'\in \mathcal{A}(\uparrow^m). $$
\end{lemma}
\begin{proof}
Let $\delta$ denote the tangle in $b$ obtained from $\gamma\cdot \gamma'$ by surgery along the link $L$.
By definition, the degree $k$ part of $Z^{LMO}(B,T)$ is given by
\[ Z^{LMO}_k(B,T) =
\left( \dfrac{\iota^{L'}_k\iota^{O_n}_k\iota_k^L\left(\check{Z}(C,\gamma\cdot \gamma') \right)}
{(\iota_k(\check{Z}(U_{+}))^{\sigma^{L'\cup O_n\cup L}_+} (\iota_k(\check{Z}(U_{-}))^{\sigma^{L'\cup O_n\cup L}_-}}
\right)_k, \]
where $\iota^{O_n}_k$, respectively, $\iota^{L'}_k$ and $\iota^L_k$, denote the map $\iota_k$ applied only to the 
copies of $S^1$ corresponding to $O_n\subset \gamma\cdot \gamma'$, respectively, to $L', L\subset \gamma\cdot \gamma'$, 
as in the definition of $\S$\ref{sec:iota}.

By following (the proof of) \cite[Theorem 6.6]{LMO}, we have
\[ Z^{LMO}_k(B,T) =
\left( \dfrac{\iota^{L'}_k\iota^{O_n}_k\left(Z^{LMO}(b,\delta) \right)}
{(\iota_k(\check{Z}(U_{+}))^{\sigma^{L'\cup O_n}_+} (\iota_k(\check{Z}(U_{-}))^{\sigma^{L'\cup O_n}_-}} \right)_k, \]
and
$$\iota^{L'}_k\iota^{O_n}_k\left(Z^{LMO}(b,\delta) \right)=\left( \iota^{L'}_k 
j^{V}_k\chi^{-1}_{O_n,V}\left(Z^{LMO}(b,\delta \right)\right)_{\le k}$$
from the definition of $\iota_k$.
By Proposition \ref{proplambda},
\[ \chi^{-1}_{O_n,V} \left(Z^{LMO}(b,\delta) \right) = \chi^{-1}_{S}Z^{LMO}(b,t)\circ_S \Lambda^{n}(S,U,V)\circ_U 
\chi^{-1}_{T_{0,n},U}\hat{Z}(C,\gamma'). \]
Note that the only copies of $S^1$ in the core of the above quantity are those corresponding to the link $L'$, so that 
applying $\iota^{L'}_k$ just amounts to applying the map $\iota_k$ of $\S$\ref{sec:iota}.
Finally, note that by our assumption on the link $L$, the linking matrix of $L'\cup O_n$ is just the linking matrix of 
the tangle $(T_{0,n}\cup L')\subset \gamma'$, so $\sigma^{L'\cup O_n}_\pm=\sigma_\pm(\gamma')$ as required.
\end{proof}
\begin{rem}\label{remgluenabla}
A similar formula holds in general for the invariant $\nabla_G$ of q-tangles in cobordisms over $\Sigma$.  The 
only requirement is that such a tangle decomposes as the stacking of some q-tangle with an element of $\mathcal{T}(n)$, 
for some integer $n$ (such as $\gamma$ in Lemma \ref{lemglueLMO}).  In this case, there is a formula similar to 
(\ref{odotk}), but the Kontsevich integral is replaced with the AMR invariant $V_G$.

To illustrate, we give a diagrammatic version of the map $J_G$ of  $\S$\ref{univtangles}, which allows us to express 
$\nabla_G(M)$ in terms of $\nabla^r_G(M)$ for a homology cylinder $M$.
Recall that $T_G$ denotes the tangle $I_G\cup (K_G\cap (1_\Sigma- C))$ in $1_\Sigma- C$, cf.\ (\ref{deftgtm}).
The subtangle $m_i\cap (1_\Sigma- C)\subset T_G$ is just a copy of the tangle $T_{0,h}$, and we have the formula
\begin{equation}\label{cg}
  \nabla_G(M) =  \nabla^r_G(M)\odot T_G \in \mathcal{A}_h,
\end{equation}
\noindent for any homology cylinder $M$ over $\Sigma$.
Though it is defined more generally for cobordisms, this expresses our universal invariant $\nabla_G$ for homology cylinders
in terms of LMO since  $\nabla_G^r(M)$ can be computed in terms of the the LMO invariant of a q-tangle in a homology ball as 
in Equation \ref{nablaredn}.\end{rem}
\subsection{Diagrammatic formulas for the topological gluings}\label{pairings}
In this section, we finally give the explicit formulas for the pairing, stacking and shelling products.%
\subsubsection{Model for preferred structures} \label{models}
We begin by choosing preferred marked bordered fatgraphs in each of the  surfaces $\Sigma_{g,1}$ and $\Sigma_{0,g+1}$.
The specified fatgraphs each have the property that the greedy algorithm produces a line segment as maximal tree; such  ``linear chord diagrams'' are studied in \cite{bene}.
The first, denoted $\overline C_g$, consists of $g$ edges attached along the line interval  $T_{\overline C_g}$, 
creating $g$ isolated humps as shown in Figure \ref{fig:stcd}.
The second fatgraph, which we call a genus $g$ \emph{symplectic fatgraph}\footnote{Note that this notation differs from 
that used in \cite{BKP}.}  and denote by $C_{g}$, consists of $2g$ edges which appear along the interval $T_{C_g}$ in
$g$ isolated overlapping pairs as illustrated in Figure \ref{fig:stcd}; see Figures \ref{proofcapping} and \ref{proofstack}
for the skinny surfaces respectively associated to $\overline{C}_2$ and $C_1$.
\begin{figure}[h!]
\input{stcd.pstex_t}
\caption{Preferred  marked  fatgraphs $\overline C_g\hra \Sigma_{0,g+1}$ and $C_g\hra \Sigma_{g,1}$.}
\label{fig:stcd}
\end{figure}
We choose the standard markings of $C_g$ in $\Sigma_{g,1}$ and $\overline C_g$ in $\Sigma_{0,g+1}$ as  shown in Figure 
\ref{fig:stcd}, where we have the identification of $\Sigma_g=\Sigma_{g,1}\cup D^2$ with  the boundary of 
$H_g=\Sigma_{0,g+1}\times I$.
\subsubsection{Diagrammatic pairing, stacking and shelling}
Let $g\ge 1$ be an integer and
define the three $q$-tangles in $C$
\[ T_g=T_{0,2g}\cup L_T, \quad S_g=T_{2g,4g}\cup L_S\quad \textrm{and}\quad R_g=T_{g,3g}\cup L_R, \]
where $L_T$, $L_S$ and $L_R$ are framed links as shown Figure \ref{tangles}.
\begin{figure}[h!]
\includegraphics{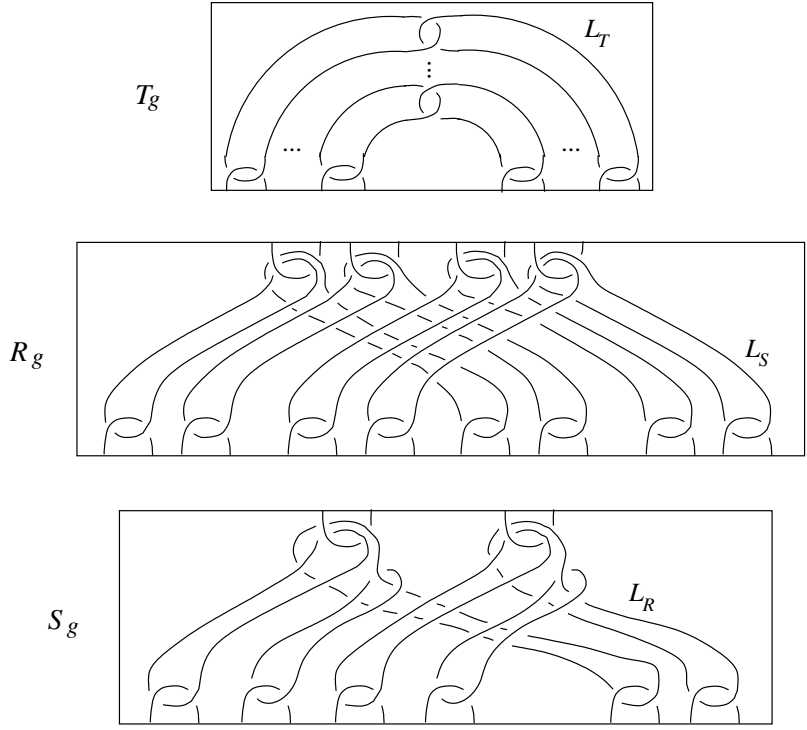}
\caption{The $q$-tangles $T_g$, $S_g$ and $R_g$.  }\label{tangles}
\end{figure}

Given Jacobi diagrams $D,D'\in \mathcal{A}_g$ and $E,E'\in \mathcal{A}_{2g}$, define
$$\aligned
 \langle \; , \; \rangle: \mathcal{A}_g \times \mathcal{A}_g &\rightarrow \mathcal{A}(\emptyset),\\
 \bullet: \mc{A}_{2g}\times\mc{A}_{2g}&\rightarrow \mc{A}_{2g},\\
 \star: \mc{A}_{2g}\times\mc{A}_{g}&\rightarrow \mc{A}_{g}\\
\endaligned
$$
by
$$\aligned 
\langle D,D'\rangle &:= (D\otimes D')\odot T_g,\\
 E\bullet E' &:= (E\otimes E')\odot S_g~~ \textrm,\\
 E\star D &:= (E\otimes D)\odot R_g.\\
\endaligned$$

\begin{theorem}\label{thmpairings}
Let $H$ and $H'$ be two genus $g$ homology handlebodies, and let $M,M'$ be two homology cylinders over $\Sigma_{g,1}$.
Then
\begin{eqnarray}
 & Z^{LMO}(H\cup_\iota H') = \langle \nabla^r_{\overline{C}_g}(H),\nabla^r_{\overline{C}_g}(H')\rangle, & \label{mTm1}
 \\
 & \nabla^r_{C_g}(M\cdot M') = \nabla^r_{C_g}(M)\bullet \nabla^r_{C_g}(M'), & \label{mTm2} \\
 & \nabla^r_{\overline{C}_g}(M\ast H) = \nabla^r_{C_g}(M)\star \nabla^r_{\overline{C}_g}(H). & \label{mTm3}
\end{eqnarray}
\end{theorem}
\begin{proof}
Let $K,K'$ be framed links in $1_{\Sigma_{0,g+1}}$ such that $H=(1_{\Sigma_{0,g+1}})_K$ and 
$H'=(1_{\Sigma_{0,g+1}})_{K'}$, and let $L,L'$ be framed links in $1_{\Sigma_{g,1}}$ such that $M=(1_{\Sigma_{g,1}})_L$ 
and $M'=(1_{\Sigma_{g,1}})_{L'}$.

Denote by $\overline{K_g}$ and $K_g$, the system of linking pairs in rigid position in $1_{\Sigma_{0,g+1}}$ and 
$1_{\Sigma_{g,1}}$, respectively,  induced by the preferred marked bordered  fatgraphs $\overline{C}_g$ and $C_g$ 
defined in $\S$\ref{models}.
Let $K_0$, $K_0'$, $L_0$, and $L_0'$ be the reduced representatives of $K$, $K'$, $L$ and $L'$ respectively with 
respect to the linking pairs $\overline{K_g}$ and $K_g$ as provided by Lemma \ref{reducedform}.
See the left-hand side of Figures \ref{proofcapping} and \ref{proofshell}.
Note that surgery along these links in the preferred box always gives a homology ball since we are considering homology 
cylinders.

As to Equation (\ref{mTm1}),
it follows from straightforward Kirby calculus that the integral homology sphere $H\cup_\iota H'$ is obtained from 
$S^3$ by surgery along the framed link depicted on the right-hand side of Figure \ref{proofcapping}.
\begin{figure}[!h]
\includegraphics{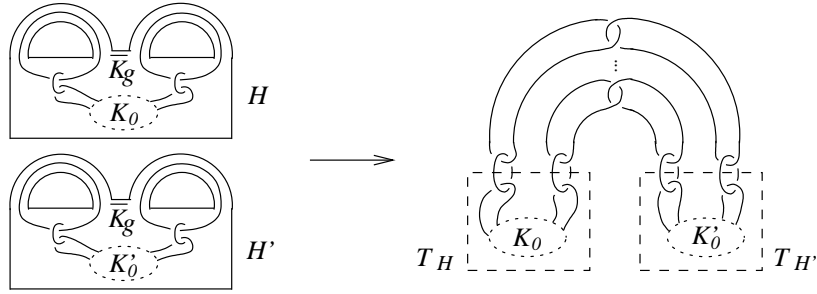}
\caption{Links for the pairing in the case $g=2$.} \label{proofcapping}
\end{figure}
We see that this link can be decomposed as $(T_H\otimes T_{H'})\cdot T_g$, where $T_H$ and $T_{H'}$ are the q-tangles 
in $C$ defined in (\ref{deftgtm}).
The tangles $\gamma=T_H\otimes T_{H'}$ and $\gamma' = T_g$ indeed satisfy the hypotheses of Lemma \ref{lemglueLMO}, 
from which the result follows.

As to Equation (\ref{mTm2}),
the stacking product $M\cdot M'$ is obtained from $1_{\Sigma_{g,1}}$ by surgery along $L\cup L'$, where $L$ and $L'$ 
respectively occur in the lower and upper half of $1_{\Sigma_{g,1}}$.
By Lemma \ref{reducedform}, we can use the system of linking pairs $K_g$ in rigid position to obtain a reduced 
representative, as
shown in the left-hand side of Figure \ref{proofstack}.

\begin{figure}[!h]
\begin{center}
\includegraphics{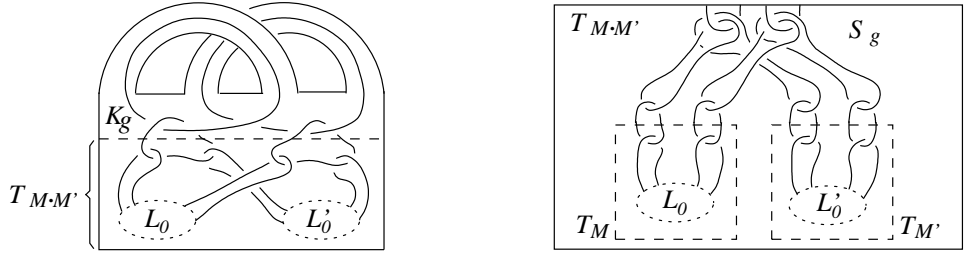}
\caption{Link and tangle for the stacking product in the case $g=1$.} \label{proofstack}
\end{center}
\end{figure}
\noindent Following (\ref{deftgtm}), denote by $T_{M\cdot M'}$ the q-tangle in $C$ obtained by cutting 
$1_{\Sigma_{g,1}}$ along the cutting face of the preferred box.  This tangle, shown on the left-hand side of Figure 
\ref{proofstack}, is Kirby equivalent to the tangle shown on the right-hand side of the figure, which can be decomposed 
as $(T_M\otimes \check T_{M'})\cdot S_g$.
The result then follows from Lemma \ref{lemglueLMO} with $\gamma=T_M \otimes \check T_{M'}$ and $\gamma' = S_g$.

Finally for Equation (\ref{mTm3}),
consider the link in $1_{\Sigma_{0,g+1}}$ obtained from $L$ and $K$ under the shelling product $1_{\Sigma_{g,1}}\star 
1_{\Sigma_{0,g+1}}$.
As in the previous case, we can use Lemma \ref{reducedform} and the system of linking pairs $\overline K_g$ in 
$1_{\Sigma_{0,g+1}}$ to  to obtain a reduced representative.
One can check using Figure \ref{figshell} and Kirby calculus that the tangle $T_{M\ast H}$ obtained by cutting 
$1_{\Sigma_{0,g+1}}$ along the cutting face of the preferred box is the tangle represented on the right-hand side of 
Figure \ref{proofshell}.
\begin{figure}[!h]
\begin{center}
\includegraphics{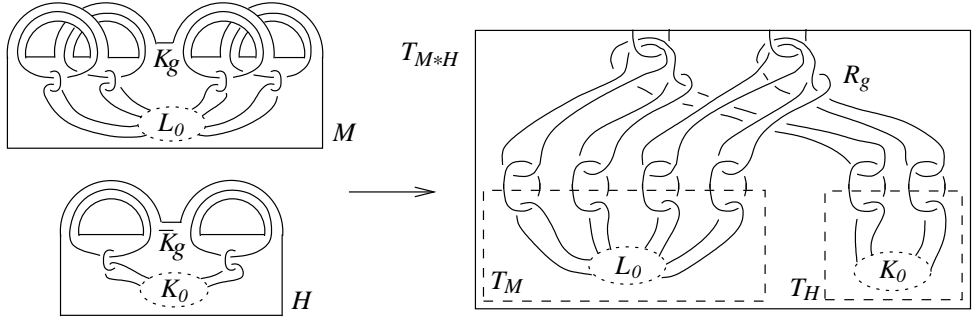}
\caption{Link and tangle for the shelling product in the case $g=1$. } \label{proofshell}
\end{center}
\end{figure}
Since the latter decomposes as $(T_M\otimes T_H)\cdot R_g$, and $\gamma=T_M\otimes T_H$ and $\gamma' = R_g$ satisfy the 
hypotheses of Lemma \ref{lemglueLMO}, the result follows.
\end{proof}
\begin{rem}
We can use (\ref{mTm2}) to define a multiplicative version of the rigid $\nabla_G^r$ invariant.  This is done by 
renormalizing the invariant by a factor which uses the tangle $R_g$ and a relative version of the contraction product 
$\odot$.  A similar renormalization appears in $\S$4.4 of \cite{CHM}.
\end{rem}

\section{Ptolemy Groupoid Representations}\label{sec:ptolemy}
In this section, we exploit the dependence of our invariant $\nabla_G$ on the fatgraph $G$
to construct representations of mapping class groups and their subgroups.

\subsection{Classical actions of subgroups of the mapping class group}\label{actions}
The {\it mapping class group} $MC(\Sigma)$ of a compact orientable surface $\Sigma$, possibly with boundary $\partial 
\Sigma$ non-empty, is the group of isotopy classes relative to $\partial \Sigma$  of orientation-preserving self-diffeomorphisms
of $\Sigma$ which fix $\partial \Sigma$ pointwise.
$MC(\Sigma)$ acts naturally on the integral homology groups of $\Sigma$, and we define the {\it Torelli group} 
$\mc{I}(\Sigma)$ of $\Sigma$ to be the subgroup of $MC(\Sigma)$ acting trivially.
Given a Lagrangian subspace $\Lambda \subset H_1(\Sigma_g;\bQ)$ for  a closed surface $\Sigma_g$, we define the {\it 
Lagrangian preserving mapping class group}
$MC(\Lambda)=\{ \varphi\in MC(\Sigma_g):\varphi_*(\Lambda)=\Lambda\}$.
In particular, it is not difficult to see that the Torelli group is the intersection of all the Lagrangian preserving 
mapping class groups.

Consider the standard Heegaard decomposition of $S^3=H_g\cup_\iota H_g$, where $\iota$ is the orientation-reversing 
involution.   Any mapping class $f\in \mc{I}(\Sigma_{g,1})$ gives rise to a corresponding mapping class $f\in 
\mathcal{I}(\Sigma_{g})$ by capping off and extending by the identity.  We may construct the homology sphere
$S^3_f=H_g\cup_{\iota\circ f} H_g$ by re-gluing the handlebodies using $\iota\circ f$.  More generally
for any Heegaard decomposition of a homology $3$-sphere $M=H\cup_\iota H'$ into two genus $g$ homology handlebodies, we 
obtain a similar map
\begin{equation}\label{eq:Mf}
f\mapsto M_f=H\cup_{\iota\circ f } H'.
\end{equation}
Composing with the LMO invariant of the resulting homology 3-sphere $M_f$, we obtain a map $\mc{I}(\Sigma_{g,1})\ra 
\mc{A}(\emptyset)$, which is of some importance \cite{morita-casson}.

This kind of action of the  Torelli group on the set of integral homology spheres with Heegaard splitting can 
equivalently be described in the context of homology cylinders via the mapping cylinder construction and the 
topological products described in Section \ref{stdpairings}.
Indeed, the mapping cylinder of $\varphi\in MC(\Sigma)$, denoted  $C(\varphi)=(1_\Sigma,\varphi,Id)$, is a special case 
of cobordism over $\Sigma$, and restricting to $\varphi\in  \mc{I}(\Sigma)$, we obtain a
homomorphism of monoids
\begin{equation}\label{eq:TorHCyl}
\aligned \mc{I}(\Sigma)&\ra \mc{HC}(\Sigma),\\
\varphi&\mapsto C(\varphi).\\
\endaligned
\end{equation}
Using this construction, we may reformulate \eqref{eq:Mf} as $f\mapsto M_f= H \cup_\iota (C(f)  \ast   H') $, thus 
making precise the sense in which \eqref{eq:TorHCyl} describes an action on the set of integral homology spheres.

More generally, we can view the homomorphism \eqref{eq:TorHCyl} as an action of $\mc{I}(\Sigma)$
on the vector space generated by homology cylinders over $\Sigma$ by stacking, i.e., $$M\mapsto M\cdot  C(\varphi) $$ 
for $M\in \mc{HC}({\Sigma})$ and $\varphi\in \mc{I}(\Sigma)$.
Similarly, we have the conjugation action  $$M\mapsto C(\varphi)\cdot M \cdot C(\varphi^{-1})$$ of $\varphi\in MC(\Sigma)$ on 
homology cylinders over $\Sigma$, where if $M=(1_\Sigma)_L$ is a homology cylinder over $\Sigma$, then
\begin{equation}\label{arewegonnafinishthissomeday}
C({\varphi})\cdot (1_\Sigma)_L\cdot C(\varphi^{-1})=C(\varphi)\cdot C(\varphi^{-1}) \cdot
(1_\Sigma)_{\varphi^{-1}(L)}=(1_\Sigma)_{\varphi^{-1}(L)}\in \mathcal{HC}(\Sigma).
\end{equation}
Analogously, we have a shelling action $$H\mapsto C(\varphi)\ast H$$ of the Lagrangian preserving subgroup 
$MC(\Lambda)$  on the  set $V(\Sigma_g, \Lambda)$ of  genus $g$ homology handlebodies with Lagrangian $\Lambda$.

Recall that the preferred marked bordered fatgraphs $C_g$ and $\overline{C}_g$ defined in $\S$\ref{models} induce 
isomorphisms
\[ \nabla_{C_g}\colon \overline{\mc{H}}_{\Sigma_{g,1}}\xra{\cong} \mc{A}_{2g}
\quad \textrm{and} \quad
 \nabla_{\overline C_g}\colon V(\Sigma_g,\Lambda^{st})=\overline{\mc{H}}_{\Sigma_{0,g+1}}\xra{\cong} \mc{A}_g, \]
 and using these, we thus obtain representations
\[
\xi\colon MC(\Sigma_{g,1})\ra \Aut(\mc{A}_{2g})
\quad \textrm{and} \quad
\zeta: MC(\Lambda^{st})\ra \Aut(\mc{A}_{g}) \]
respectively induced by conjugation and the shelling action.
This section relies on the fundamental relationship between fatgraphs and mapping class groups provided by the Ptolemy 
groupoid of decorated \Teich theory to describe these various actions in a purely combinatorial way.
\subsection{Ptolemy Groupoid}\label{ptolemy}
We shall restrict for convenience to surfaces with only one boundary component.
Given a bordered fatgraph $G$,
define the  \emph{Whitehead move} $W$ on a non-tail edge $e$ of the uni-trivalent fatgraph $G$ to be the  modification 
that collapses $e$ to a vertex of valence four and then expands this vertex in the unique distinct way to produce the 
uni-trivalent fatgraph $G'$; see Figure \ref{fig:WMove}. We shall write
either $W\colon G\ra G'$ or $G\xra{W} G'$ under these circumstances.

Not only do markings of fatgraphs evolve in a natural way under Whitehead moves, so that we can unambiguously speak of 
Whitehead moves on marked fatgraphs, but also there is a natural identification of the edges of $G$ and $G'$.  
Furthermore, there are three families of finite sequences of Whitehead moves, called the involutivity, commutativity, 
and pentagon relations, which leave invariant each marked
fatgraph $G$, cf.\ \cite {penner93,MP}.
\begin{figure}[h!]
\includegraphics[width=3 in]{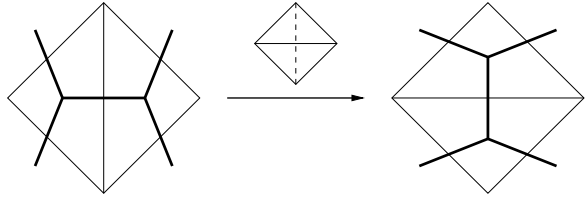}
\caption{A Whitehead move of a fatgraph superimposed on top of the diagonal exchange of its dual triangulation. }
\label{fig:WMove}
\end{figure}

The \emph{Ptolemy groupoid} $\Pt(\Sigma)$ of a surface $\Sigma$ with boundary is defined as the groupoid with objects 
given by marked bordered uni-trivalent fatgraphs $G\hra\Sigma$ and morphisms given by sequences of Whitehead moves 
modulo the involutivity, commutativity, and pentagon relations.\footnote{The term ``Ptolemy groupoid'' is sometimes 
used to refer to the groupoid whose objects are $MC(\Sigma)$-orbits of uni-trivalent fatgraphs and whose morphisms are
$MC(\Sigma)$-orbits of pairs of such with the natural composition.  We prefer to call this the \emph{mapping class
groupoid} since it gives a combinatorial model for the fundamental path groupoid of Riemann's moduli space. }

$\Pt(\Sigma)$ provides a $MC(\Sigma)$-equivariant combinatorial model of the fundamental path groupoid of the decorated 
\Teich space of $\Sigma$, cf.\ \cite{penner, penner93,Penner04}.
As such, given any ``point'' in $\Pt(\Sigma)$, i.e., any marked bordered fatgraph $G\hra \Sigma$, each mapping class 
$\varphi\in MC(\Sigma)$ is represented by a unique morphism from $G$ to $\varphi (G)$ in $\Pt(\Sigma)$, where $\varphi 
(G)$ is the marked fatgraph that arises by postcomposing the marking $G\hra\Sigma$ of $G$ with $\varphi$.

For any marked bordered fatgraph  $G\hra\Sigma$, we may thus think of $MC(\Sigma)$ as being a set of equivalence 
classes of paths beginning at the point $G\hra\Sigma$ and ending at a fatgraph combinatorially equivalent to $G$ but potentially
with a different marking in $\Sigma$.
 In this way, we get a presentation of the mapping class group of $\Sigma$:
\begin{theorem}[\cite{Penner04}]
For a surface $\Sigma$ with boundary, the mapping class group $MC(\Sigma)$ has a presentation with generators given by sequences of Whitehead moves on marked 
fatgraphs in $\Sigma$  beginning and ending at combinatorially isomorphic fatgraphs.  The relations in this groupoid 
are given by identifying two sequences if they differ by a finite number of insertions or deletions of involutivity, 
commutativity, and pentagon relations.
\end{theorem}
In a similar way,  the Torelli group $\mathcal{I}(\Sigma)$ (and indeed each term of the Johnson filtration 
\cite{johnson}) likewise admits an analogous combinatorial presentation as in \cite{MP}.

By a \emph{representation}  $\Pt(\Sigma)\to K$ of the Ptolemy groupoid in some group $K$, we mean a composition-preserving map $\Mor(\Pt(\Sigma))\ra K$ from 
the morphisms of $\Pt(\Sigma)$.
In other words, a representation of $\Pt(\Sigma)$ is a morphism that assigns an element of $K$ to each Whitehead move 
such that the composition is trivial for the involutivity, commutativity, and pentagon relations.
\subsection{The explicit Ptolemy action on $\mc{A}_h$.}\label{explicitlift}
Our first representation of $\Pt(\Sigma)$ captures the dependence of $\nabla_G$ on the choice of marked bordered 
fatgraph $G\hra \Sigma$ giving a representation as automorphisms of an appropriate space of Jacobi diagrams extending 
the conjugation action of the mapping class group on homology cylinders.   Recall from Corollary \ref{cor:iso} that for 
a genus $g$ surface $\Sigma$ with $n\ge 1$ boundary components, any marked bordered fatgraph $G\hra\Sigma$ provides a 
graded isomorphism
\[ \nabla_G\colon \overline{\mc{H}}_\Sigma \xra{\cong} \mc{A}_{h}, \]
\noindent where $h=2g+n-1$, as a consequence of the universality of the invariant $\nabla_G$.
Thus, for any marked bordered fatgraphs $G$ and $G'$, we get isomorphisms $\nabla_G$ and $\nabla_{G'}$ of $ 
\overline{\mc{H}}_\Sigma$ with $\mc{A}_{h}$.  As a formal consequence, we obtain an explicit representation of the
Ptolemy groupoid:
\begin{theorem}
The map
\[ (G\xra{W} G') \mapsto \nabla_{G'} \circ\nabla_G\inv. \]
defines a representation of the  Ptolemy groupoid acting on  $\mc{A}_{h}$
 \[ \hat \xi\colon \Pt(\Sigma)\ra \Aut(\mc{A}_{h}). \]
For $\Sigma=\Sigma_{g,1}$, this representation extends the representation $\xi\colon MC(\Sigma_{g,1})\ra 
\Aut(\mc{A}_{2g})$ induced by the conjugation action in the sense that for any sequence of Whitehead moves
$C_g\xra{W_1}\dotsm\xra{W_k} \varphi(C_g)$ representing $\varphi\in MC(\Sigma_{g,1})$, we have the identity 
$\hat\xi(W_1)\circ\dotsm \circ\hat\xi(W_k)=\xi(\varphi)$.
\end{theorem}

Before giving the proof, we first give the following topological interpretation of the automorphism 
associated to a morphism from $G$ to $G'$ in $\Pt(\Sigma)$.  
Given an element in $\mc{A}_h$, we can pull it back via $\nabla_G\inv$ to an element of $\overline{\mc{H}}_\Sigma$,
represented by a formal series $L$ of framed links in $1_\Sigma$ in admissible position with respect to the polygonal
decomposition $P_G$.  We then evolve $G$ by a sequence of Whitehead moves to a new marked fatgraph $G'$, and isotope
the links in $L$ accordingly to put them in admissible position with respect to the new  polygonal decomposition
$P_{G'}$.
Evaluating $\nabla_{G'}$ on the resulting series of links then provides a new element of $\mc{A}_h$.

\begin{proof}
The fact that the above action defines a representation of the Ptolemy groupoid follows easily since any sequence of 
Whitehead moves representing a trivial morphism of $\Pt(\Sigma)$ begins and ends at identical marked fatgraphs and thus  
must give the trivial action.

For any  link $L\subset 1_\Sigma$, we have $\nabla_n^G(L)=\nabla_n^{\varphi(G)}(\varphi(L))$
by construction, so that
\[ \nabla_{\varphi(G)}((1_\Sigma)_L)=\nabla_{G}((1_\Sigma)_{\varphi^{-1}(L)}). \]
Thus by (\ref{arewegonnafinishthissomeday}) for any $M\in \overline{\mc{H}}_\Sigma$,
we have
\[
\xi(\varphi)(\nabla_G(M))= \nabla_G(C({\varphi})\cdot M\cdot C(\varphi^{-1}))=\nabla_{\varphi(G)}(M),
\]
and setting $G=C_g$, the result follows.
\end{proof}
\subsection{The Ptolemy groupoid action on handlebodies}\label{sect:hb}
A similar extension of the shelling action $\zeta$ arises as follows.
By Lemma 7.4 of \cite{ABP}, there is an algorithm which produces a representation 
$$\Pt(\Sigma_{g,1})\ra MC(\Lambda^{st})$$ of the Ptolemy groupoid extending the identity on $MC(\Lambda^{st})$.  Thus, 
we obtain:
\begin{proposition}
Let $\Sigma_g$ be a closed genus $g$ surface.
Fix a disc in $\Sigma_g$ and let $\Sigma_{g,1}$ be its complement.  Then we have an  explicit algorithmically defined 
representation
 \[ \hat \zeta: \Pt(\Sigma_{g,1})\ra \Aut(\mc{A}_g) \]
which extends the shelling action $\zeta: MC(\Lambda^{st})\ra \Aut(\mc{A}_g)$.
\end{proposition}
Owing to its dependence on the complicated
algorithms in \cite{ABP}, the action on $\mc{A}_g$ obtained in this way is more complicated  than the action on $\mc{A}_{h}$ described in the previous section.
\subsection{Extension of the LMO invariant to the Ptolemy groupoid}\label{sec:liftLMO}
In this section, we give a kind of Ptolemy groupoid action on finite type invariants of integral homology spheres which 
extends the usual action of the Torelli group via Heegaard decomposition.
More precisely, we give a Ptolemy groupoid action on finite type invariants of homology cylinders over $\Sigma$ which 
extends the stacking action of $\mc{I}(\Sigma)$ on $\overline{\mc{H}}_\Sigma$, and which in the case of 
$\Sigma=\Sigma_{g,1}$ induces a map from $\Pt(\Sigma_{g,1})$ to $\mc{A}(\emptyset)$  extending the analogous map of the 
Torelli group $\mc{I}(\Sigma_{g,1})$.  

We begin by recalling  Corollary 7.1 and Theorem 8.1 of \cite{ABP}, which together
give\footnote{The proof in \cite{ABP} relies on a sequence of algorithms, beginning with the greedy algorithm, which 
produces a sequence of Whitehead moves taking a given fatgraph to a symplectic one, followed by an algorithm which
manipulates the homological information associated to each edge of the symplectic fatgraph; this last algorithm
apparently has a paradigm in K-theory. }
that any choice of a  marked bordered fatgraph  $G\hra \Sigma_{g,1}$ determines a representation
\[ \hat{id}_G\colon \Pt(\Sigma_{g,1})\ra \mc{I}(\Sigma_{g,1}) \]
of the Ptolemy groupoid which extends the identity homomorphism of $\mc{I}(\Sigma_{g,1})$.  Let 
$\hat{id}=\hat{id}_{C_g}$ be the representation
provided by  the marked fatgraph $C_g\hra \Sigma_{g,1}$ defined in $\S$\ref{models}.
Define a representation of the Ptolemy groupoid of $\Sigma_{g,1}$
\[
\rho: \Pt(\Sigma_{g,1})\ra\mc{A}_{2g},
\]
to be a composition-preserving map, where the target space is imbued with the stacking product $\bullet$, by setting
\[ \rho(W) := \nabla^r_{C_g}\left( C(\hat{id}(W)) \right). \]

\begin{theorem}\label{thm:lmo}
The representation $\rho: \Pt(\Sigma_{g,1})\ra\mc{A}_{2g}$ of the Ptolemy groupoid of $\Sigma_{g,1}$ provides  an 
extension of the LMO invariant of integral homology spheres to the Ptolemy groupoid in the following sense:
Let $f\in \mathcal{I}_{g,1}$ and let
 $$ G\xra{W_1} G_1\xra{W_2}...\xra{W_k}G_k=f(G) $$
be a sequence of Whitehead moves representing $f$.  Let $M=H\cup_\iota H'$ be a genus $g$ Heegaard splitting of an 
integral homology sphere $M$.
Then the LMO invariant of the integral homology $3$-sphere $M_f=H\cup_{\iota\circ f} H'$ is given by
\[  Z^{LMO}(M_f)=\left\langle v , \left(\rho(W_1)\bullet \rho(W_2)\bullet \dotsm  \bullet 
\rho(W_k)\right)\star v' \right\rangle, \]
where $v=\nabla^r_{\overline C_g}(H)\in \mc{A}_g$ and $v'=\nabla^r_{\overline C_g}(H')\in \mc{A}_g$.
\end{theorem}
\begin{proof}
Since $\hat{id}$ extends the identity homomorphism of $\mc{I}(\Sigma_{g,1})$, we therefore have $\hat{id}(W_1)\circ 
\hat{id}(W_2)\circ ...\circ \hat{id}(W_k) = f$, hence
$M_{W_1}\cdot M_{W_2}\cdot ...\cdot M_{W_k}=C(f)$, where $M_{W}$ denotes $C(\hat {id} (W))$.
Since $M_f=H\cup_\iota (C(f)\ast H')$, the formula follows from Theorem \ref{thmpairings}.
\end{proof}

Considering the map $f\mapsto S^3_f$ induced by the standard Heegaard decomposition of $S^3$,
Theorem \ref{thm:lmo} shows that for a sequence $G\xra{W_1} G_1\xra{W_2}...\xra{W_k}G_k=f(G) $ of Whitehead moves 
representing $f$, we have
\[  Z^{LMO}((S^3)_f)=\left\langle v_0 , \left(\rho(W_1)\bullet \rho(W_2)\bullet \dotsm  \bullet \rho(W_k)\right)\star 
v_0 \right\rangle, \]
where the diagrammatic constant $v_0=\nabla^r_{\overline C_g}(H_g)\in \mc{A}_g$ can easily be computed as follows.
By definition, we have $\nabla^r_{\overline C_g}(H_g)=Z^{LMO}(C,T_{g,0})$, where $T_{g,0}$ is the q-tangle of Figure 
\ref{tcap}.  By (\ref{Z1}), the Kontsevich integral $Z(T_{g,0})\in \mathcal{A}(\uparrow^g)$ of this tangle is thus 
given by including a $\sqrt{\nu}$ on each copy of $\uparrow$.  It follows that
\[ v_0=\sqcup_{i=1}^g (\chi^{-1}_{\{i\}}\sqrt{\nu})\in \mathcal{A}_g, \]
where an explicit formula for $\nu$ is given in \cite{BNGRT2}.
\subsection{Extension of the first Johnson homomorphism} \label{sec:lifttau1}
In \cite{MP}, a representation of the Ptolemy groupoid was introduced using the notion of an $H$-marking of a fatgraph 
$G$ and shown to be an extension of the first Johnson homomorphism $\tau_1$ to the Ptolemy groupoid.
In this section, we show how a variation of the invariant $\nabla_G$ can be used to realize this extension.
\subsubsection{General latches}\label{general}
Let $G\hra\Sigma$ be a marked bordered fatgraph in a surface $\Sigma$.  We begin by introducing a generalized notion of 
a system of latches in $1_\Sigma$ and thus of the invariant $\nabla_G$.  In fact, the main property of the system of 
latches $I_G$ we used so far in this paper, besides the fact that it is determined by the fatgraph $G$, is that it 
provides a dual basis in homology.

We define a general latch for $G$ as an embedded interval in the boundary of $\Sigma \times I$ with endpoints lying in 
$(\partial \Sigma)\times \{ \frac{1}{2} \}$ such that it can be isotoped relative to its boundary to be in admissible 
position with respect to the polygonal decomposition $P_G$.  A collection of $h$ disjoint latches in the boundary of 
$1_\Sigma$ whose homotopy class relative to the boundary induces a free basis for $H_1(1_\Sigma,\partial 
1_\Sigma;\mathbb{Q})$ is a \emph{general system of latches} for $G$.

It is clear that substituting for $I_G$ in (\ref{nabla+}) any general system of latches yields an invariant of cobordisms.  
In fact,  such an invariant is also universal for homology cylinders.  We shall not make use of this result 
and omit the proof, which essentially follows $\S$\ref{sec:proofuniversalitynabla} (the main difference being in the 
definition of the surgery map).
\subsubsection{Extending $\tau_1$ via the invariant $\nabla$}
We restrict our attention to the once-bordered surface $\Sigma=\Sigma_{g,1}$ of genus $g$, set 
$H=H_1(\Sigma_{g,1};\mathbb{Z})$ and $H_\mathbb{Q}=H\otimes \mathbb{Q}$.  Recall \cite{johnson} that the first Johnson 
homomorphism
 \[ \tau_1\colon \mc{I}(\Sigma_{g,1}) \ra \Lambda^3 H \]
takes its values in the third exterior power of $H$.

Denote by $I_g$ the $2g$-component q-tangle in $\Sigma_{g,1}\times I$ represented below.
 \[ \textrm{\includegraphics{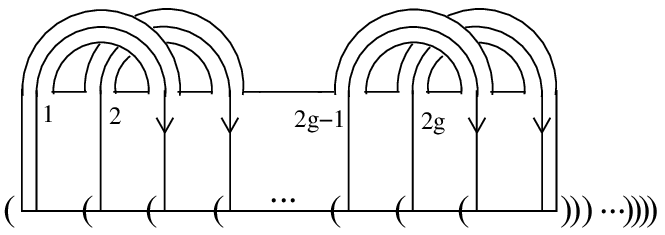}} \]
Note that by isotoping $I_g$  so that it is  contained in $(\Sigma_{g,1}\times \{1\})\cup(\partial\Sigma_{g,1}\times 
I)$, we may consider $I_g$ as a general system of latches for any choice of marked bordered fatgraph $G$ in
$\Sigma_{g,1}$.
Indeed, one can unambiguously  arrange the endpoints of $I_g$ so  that under the projection of $\Sigma_{g,1}\times I$ 
to $\Sigma_{g,1}$ they  lie in a neighborhood of the fixed point $q$ where the tail of $G$ is attached, so  that 
$\partial I_g$ lies on the boundary of the preferred box in the polygonal decomposition $P_G$.

Let $G$ be a marked bordered fatgraph in $\Sigma_{g,1}$ and let $L$ be a framed link in $\Sigma_{g,1}\times I$ which is 
disjoint from both $K_G$ and $I_g$.  
  Set
\begin{equation}\label{nablat}
\nabla^{G,I_g}_n(L):= \frac{\iota_n(\check{V}_G(L\cup K_G\cup I_g))}
{ \iota_n(\check{V}_G(U_{+}))^{\sigma^{L\cup K_G}_+} \iota_n(\check{V}_G(U_{-}))^{\sigma^{L\cup K_G}_-} }\in
\mathcal{A}_{\le n}(\uparrow^{2g}),
\end{equation}
where we make use of the notation of (\ref{nabla+}).  This quantity is an invariant of the surgered manifold 
$M=(\Sigma_{g,1}\times I)_L$, and following (\ref{nabla}), we set
\[ \nabla^{I_g}_G(M):= 1+\big( \nabla^{G,I_g}_1(L) \big)_{1} + ... + \big(\nabla^{G,I_g}_n(L) \big)_{n} +... \in 
\mathcal{A}(\uparrow^{2g}). \]

Next, consider a Whitehead move $W\colon G\mapsto G'$.  We can then compare the value of the invariants 
$\nabla^{I_g}_G$ and $\nabla^{I_g}_{G'}$ on the trivial element $1_{\Sigma_{g,1}}$ and assign the quotient to the 
Whitehead move $W$ to define a map
 \[ \mc{J}(W)= \nabla^{I_g}_{G'}(1_{\Sigma_{g,1}}) / \nabla^{I_g}_G(1_{\Sigma_{g,1}})\in \mathcal{A}(\uparrow^{2g}).   
 \]
More generally, for any two marked fatgraphs in $\Sigma_{g,1}$, not necessarily related by a Whitehead move, we can 
similarly take the quotient, and in the case that these two fatgraphs are equal, we get a trivial contribution by 
definition.  This guarantees that this map $\mc{J}$ is the identity for the involutivity, commutativity, and pentagon 
relations, and hence we obtain a representation
  \[ \mc{J}\colon \Pt(\Sigma_{g,1})\ra  \mc{A}(\uparrow^{2g}). \]

  Recall that the groups $H_1(\Sigma_{g,1};\bQ)$ and $H_1(\Sigma_{g,1},\partial\Sigma_{g,1};\bQ)$ are isomorphic via 
  \Poin duality.   Define a map ${\mathfrak h}:\{1,...,2g\}\rightarrow H$ by taking $i$ to the element of $H$ dual to 
  the class of the $i$th component of $I_g$ in  $H_1(\Sigma_{g,1},\partial\Sigma_{g,1};\bZ)$.  More concretely, if we  
  let $\{A_i,B_i\}_{i=1}^{2g}$ denote the standard symplectic basis of $\Sigma_{g,1}$ with $A_i\cdot B_j =\delta_{ij}$, 
  then
   \[
{\mathfrak h}(2k)=A_k,  \quad  {\mathfrak h}(2k-1)=B_k \quad  \textrm{ for } k=1,...,g.
\]
Also recall that $\mathcal{B}^\mathsf{Y}_1(2g)=\mathcal{B}^\mathsf{Y}_{1,3}(2g)$ is the space of $2g$-colored $\mathsf{Y}$-shaped Jacobi 
diagrams and that we have the well-known and elementary isomorphism $\mc{B}_1^\mathsf{Y}(2g)\cong \Lambda^3 H_\mathbb{Q}$ 
defined by sending a $\mathsf{Y}$-shaped diagram colored by $i,j,k$ (following the vertex-orientation) to ${\mathfrak 
h}(i)\wedge {\mathfrak h}(j)\wedge {\mathfrak h}(k)\in\Lambda ^3 H$.

In order to extend the first Johnson homomorphism $\tau_1$, we restrict the target of our representation $\mc{J}$ by 
composing it with the series of maps given by
 \begin{equation}\label{eq:chainofisos}
Y: \mc{A}(\uparrow^{2g})\ra \mc{B}(2g)\ra \mc{B}^\mathsf{Y}(2g)\ra\mc{B}_1^\mathsf{Y}(2g)\cong\Lambda^3 H_\mathbb{Q}.
 \end{equation}
 From this, we
 to obtain a representation of the Ptolemy groupoid
$$ \mc{J}^\mathsf{Y}\colon \Pt(\Sigma_{g,1})\ra  \Lambda^3 H_{\mathbb{Q}}. $$
The first map in (\ref{eq:chainofisos}) is the inverse $\chi^{-1}$ of the Poincar\'e--Birkhoff--Witt isomorphism, and 
the second and third maps are the natural projections.

\begin{theorem}\label{lifttau1}
The representation $\mc{J}^\mathsf{Y}$ extends the first Johnson homomorphism $\tau_1$ to the Ptolemy groupoid.
More precisely, given a sequence
 $$ G\xra{W_1} G_1\xra{W_2}...\xra{W_k}G_k=\varphi(G) $$
of Whitehead moves representing $\varphi\in \mathcal{I}_{g,1}$, we have $\tau_1(\varphi) = 4\sum_{i=1}^k 
\mc{J}^\mathsf{Y}(W_i)$.
\end{theorem}
\subsubsection{Proof of Theorem \ref{lifttau1}}
The computation of the invariant $\mc{J}^\mathsf{Y}$ is considerably simplified by the following observation.
\begin{lemma}\label{air}
For any marked bordered fatgraph $G$ in $\Sigma_{g,1}$, we have
 $$ Y\left(\nabla^{I_g}_G(1_{\Sigma_{g,1}})\right) = Y\left( \check{V}_G(I_g) \right)\in \Lambda^3 H_\mathbb{Q}, $$
where $Y$ is the sequence of maps in (\ref{eq:chainofisos}).
\end{lemma}
\noindent In other words, the $\mathsf{Y}$-shaped part of $\nabla^{I_g}_G(1_{\Sigma_{g,1}})$ comes purely from the 
tangle $I_g$, and the system of linking pairs $K_G$ can simply be ignored in the computation.
\begin{proof}
We shall freely make use of the terminology introduced in the proof of Theorem \ref{thm:univ}.
In computing $\nabla^{I_g}_G(1_{\Sigma_{g,1}})$, we can choose $L$ to be empty in (\ref{nablat}).
By \cite[pp.\ 283]{O}, we have that $\iota_2(\check{Z}(U_\pm))=1+$terms of $i$-degree $\ge 2$, and it follows that
 $$ Y\left(\nabla^{I_g}_G(1_{\Sigma_{g,1}})\right) = Y\left( \iota_2(\check{V}_G(K_G\cup I_g)) \right). $$

We now consider the linking pairs $K_G$. We may assume that there are $2g$ disjoint $3$-balls in $1_{\Sigma_{g,1}}$ 
that intersect the system $K_G$ of linking pairs as illustrated on the right-hand side of Figure \ref{main}.  The 
Kontsevich integral of the tangle contained in these balls is computed in \cite[Theorem 4]{BNGRT2}, from which it 
follows that the only terms in $\check{V}_G(K_G\cup I_g)$ that can contribute to 
$Y\left(\nabla^{I_g}_G(1_{\Sigma_{g,1}})\right)$ have exactly $4$ vertices on each meridian core, which are the ends of 
$4$ parallel struts connecting each to the corresponding longitude core.

Suppose that some longitude core has $k$ additional vertices attached.  It follows from the definition that applying 
the map $\iota_2$ produces $\ge k$ univalent vertices, which imposes the constraint that $k\le 1$.
For $k=1$, the diagram is also sent to zero by the map $\iota_2$ since we obtain a sum of Jacobi diagrams each having a 
looped edge, which vanish by the AS relation.  Thus, the only terms which can possibly contribute are Siamese diagrams 
with $4$ struts, cf.\ Figure \ref{AMRclasper}, which come with a coefficient $\frac{1}{4!}$.
As seen in $\S$\ref{sec:proofunivnabla}, $\iota_2$ maps each Siamese diagram to a factor $(-1)^2 4!$ as required.
\end{proof}

We can now proceed with the proof of Theorem \ref{lifttau1} and  calculate the representation $\mc{J}^\mathsf{Y}$ on a Whitehead 
move $W$.
To this end, for any marked fatgraph $G$ in $\Sigma_{g,1}$, we can assume by Lemma \ref{comb} that the q-tangle $I_g$ 
is in admissible position and intersects each box except the preferred one in a trivial q-tangle.  For each oriented edge of $G$, we may equip each strand of the trivial q-tangle in the corresponding box with a sign, according to whether 
its orientation agrees (plus sign) or disagrees (minus sign) with the specified one.
For each oriented edge of $G$, assign an element of $H$ to each box except the preferred one as follows: use the map $\mathfrak{h}$ to label all 
the strands of $I_g$ intersecting the box by elements of the symplectic basis $\{A_i,B_i\}_{i=1}^{2g}$ and take the 
signed sum of these labels in $H$.  We remark that this assignment is precisely an $H$-marking as described in \cite{MP,BKP}.

Thus, we have a situation as in the upper part of Figure \ref{wmoveamr}, where each of the three strands depicted there 
represents a collection of parallel strands of $ I_g\cup K_G$ and where $A,B,C\in H$ are the labels of the box as just 
explained.  Note that the bracketing $(C,(B,A))$ in the bottom-left box is imposed by the condition on hexagons, see 
$\S$\ref{polydec}.
 \begin{figure}[h!]
 \includegraphics[height=3in]{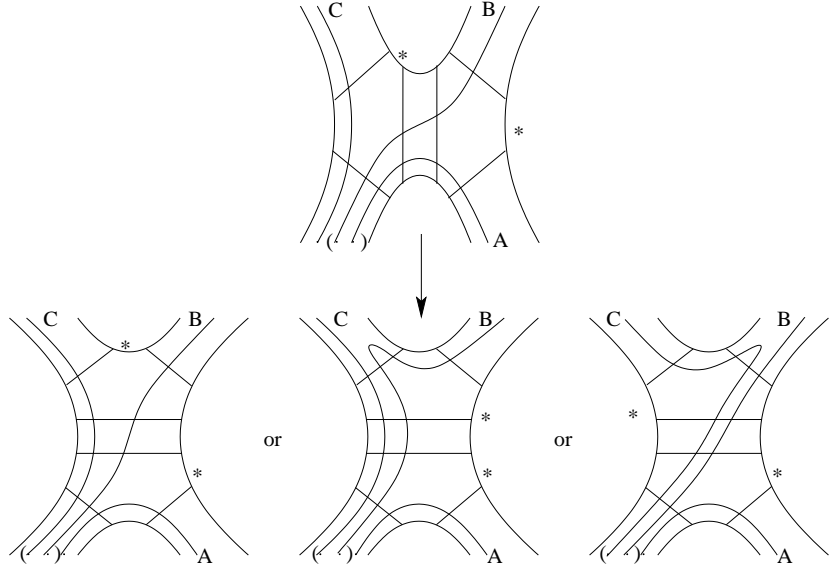}
 \caption{The three possible evolutions of the forbidden sectors under a Whitehead move. }\label{wmoveamr}
 \end{figure}
After the Whitehead move, we have one of the three situations represented in the lower part of Figure \ref{wmoveamr} 
depending on the ordering of the sectors associated to the edge on which the move has been performed.  In each case, we 
see that the bracketing of the three strands in the bottom left box is changed to $((C,B),A)$.  Also, in the last two 
cases, we get an extra cap or cup due to the evolution of the forbiden sectors, and these are the only changes; in 
particular, there are no crossing changes.

It follows from the computation \cite{BNGRT2} of $\nu$ that a cup or a cap cannot contribute to $\mc{J}^\mathsf{Y}$, so in all 
three cases, we get the same value for $\mc{J}^\mathsf{Y}(W)$ coming from the evolution in the bracketing, i.e., from the 
associator.
Recall that an even associator is always of the form
\begin{equation}\label{assocY}
\Phi=1+ \frac{1}{24}\ftt{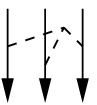}+ \textrm{terms of $J$-degree $>$ 3}.
\end{equation}
Furthermore, the value of the Kontsevich integral on a q-tangle obtained from $\Lambda_+$ (see Figure \ref{elementary}) 
by taking parallel copies of each strand with arbitrary orientation is obtained from $\Phi$ by the comultiplication and 
antipode maps defined in $\S$\ref{opjac}.

By Lemma \ref{air} and Equation \ref{assocY}, we find $\mc{J}^\mathsf{Y}(W) = \frac{1}{24} A\wedge B\wedge C\in \Lambda^3 H_{\mathbb{Q}}$, and this 
formula coincides with one fourth of the Morita-Penner extension of the first Johnson homomorphism $\tau_1$.  The fact 
that it indeed is a multiple of an extension of $\tau_1$ follows as in  \cite{MP}, upon which our determination of the 
factor $4$ currently relies, and completes the proof of Theorem \ref{lifttau1}.
\section{Concluding remarks and questions}\label{conclusions}
There are several obvious questions regarding the  Ptolemy representations derived in Section \ref{sec:ptolemy}.  Most 
notably, one may ask for a geometric interpretation of the mapping class group action arising from the representation 
$\mc{J}$.  In particular, does $\mc{J}$ provide an extension of the full LMO invariant in the same TQFT spirit as in 
Theorem \ref{thm:lmo}?

Also, a natural and interesting issue is the faithfulness of the action of the mapping class group $MC(\Sigma_{g,1})$
on $\mc{A}_{2g}$
induced by $\mc{J}$.  As the groupoid 
formulas for these representations seem simpler to analyze than their corresponding  mapping class group expressions, 
our techniques here may prove pivotal in providing such an answer.   The facts that the pronilpotent representation of 
an automorphism of a free group is faithful and that the Johnson theory presumably corresponds to the tree-like part of LMO 
by \cite{GL,Habegger,massuyeau} together suggest that the induced representation of the mapping class group may be
faithful.

\subsection{Magnus expansions and Johnson homomorphisms}
In recent beautiful work, 
 Gw\'ena\"el Massuyeau \cite{massuyeau} has introduced the notion of symplectic Magnus expansions and proved their 
 existence by giving explicit formulas in terms of the LMO invariant.
Nariya Kawazumi \cite{kawazumi-private} has asked the interesting question if such Magnus expansions might be computed 
directly in terms of suitably marked fatgraphs as in \cite{BKP}.
Our computation here of the LMO invariant provides
such a formula but again a very complicated one.
 Moreover, it seems likely that a construction analogous  to Massuyeau's  using our invariants $\nabla_G$ 
 or $\nabla_G^{I_g}$ will lead to 
 a directly computable version, and it would be an  interesting prospect  to derive formulas for the various Johnson 
 homomorphisms in terms of such a  symplectic Magnus expansion.

\subsection{Relation to triangulations of 3-manifolds}

The dual in a surface $\Sigma$ of a marked uni-trivalent fatgraph $G$ is a triangulation $\Delta_G$ of the surface 
$\Sigma$, where a $k$-valent vertex of $G$ gives rise to a $2k$-gon whose alternating sides correspond to incident 
half-edges and whose complementary sides correspond to arcs in the boundary, cf.\ \cite{penner, Penner04}.  The dual of 
a Whitehead move on a uni-trivalent fatgraph corresponds to a diagonal exchange on its dual ideal triangulation as illustrated in 
Figure \ref{fig:WMove}.  We may imagine this diagonal flip as exchanging the front and the back pair of faces of a 
tetrahedron in the obvious way.
It is thus natural to regard a morphism in the Ptolemy groupoid as a sequence of adjoined tetrahedra starting from the 
corresponding fixed ideal triangulation $\Delta$ of the surface, i.e., a morphism provides  a triangulated cobordism 
between one copy of the surface with triangulation $\Delta$ and another copy of the surface with potentially another 
triangulation.  This is especially natural for a mapping cylinder, where the Ptolemy morphism connects $\Delta$ to its 
image under the corresponding mapping class; this has indeed been the point of view in \cite{MP,BKP}.

Conversely, suppose that we have ideal triangulations of two bordered surfaces $\Sigma$ and $\Sigma'$ and suppose that 
$M$ is a 3-manifold whose boundary contains $\Sigma\sqcup\Sigma '$.  We may ask for a triangulation of $M$ extending 
those given on the boundary all of whose vertices lie
in $\Sigma\sqcup\Sigma '$.  In the spirit of a TQFT, we are led to the following questions.
Do finite compositions of Whitehead moves acting as before 
on triangulated cobordisms in fact act transitively on such triangulations of 
$M$?  Which 3-manifold invariants can be computed that depend upon the ideal triangulations of  $\Sigma\sqcup\Sigma '$ 
but not the triangulation of $M$?
What type of state-sum model corresponds to this?

\subsection{The original AMR invariant}

As pointed out in $\S$\ref{amr}, the AMR invariant $V_G$ employed in the construction of $\nabla _G$ is actually only a 
weak version of the one in \cite{AMR}.  Indeed, we are post-composing the original invariant with the map that forgets 
the homotopy class of chord diagrams on surfaces. It is a natural and important problem to try to build a $3$-manifold 
invariant from the full Andersen-Mattes-Reshetikhin invariant that would retain this homotopy information and thus 
non-trivially extend finite type invariants to all 3-manifolds.
We shall return to this study in a forthcoming paper, where we also discuss how constructions inspired by those in this 
paper can be used to define universal perturbative invariants of closed 3-manifolds and more generally universal 
perturbative TQFTs.

Finally note that in the proof of Theorem \ref{lifttau1}, computations were made amenable by Lemma \ref{air} in 
avoiding the complex maps $\iota_n$ from LMO theory, i.e., our calculation of $\tau _1$ is performed at the ``AMR level'' rather than at the ``LMO level'', cf.\ \cite{Habegger}.  The AMR-valued version of our invariant, or its homotopy 
analogue just discussed, may be suited to other explicit computations as well.
Indeed, the original AMR invariant provides a graded isomorphism between the Vassiliev-filtered free vector space 
generated by links in the cylinder over a surface with a non-empty boundary and the algebra of chord diagrams on the 
surface \cite{AMR}, and this isomorphism is determined once a suitable fatgraph is chosen in the surface as discussed here. We therefore 
get an action of the Ptolemy groupoid on the algebra of chord diagrams on any surface with non-empty boundary just as 
in $\S$\ref{explicitlift}. We shall study this representation of the Ptolemy groupoid in a forthcoming publication.

\vspace{0.5cm}
\end{document}